\documentclass[11pt]{article}
\usepackage[T1]{fontenc}
\usepackage{amsmath,amssymb}
\usepackage{latexsym, amsthm}
\usepackage{enumerate}
\usepackage{bbm}

\usepackage{authblk}
\usepackage{setspace} 
\usepackage{enumitem}
\usepackage{tikz-cd}
\usepackage{hyperref}

\large\normalsize

\theoremstyle{plain}
\newtheorem{theorem}{Theorem}[section]
\newtheorem{lemma}{Lemma}[section]
\newtheorem{proposition}{Proposition}[section]
\newtheorem{corollary}{Corollary}[section]

\theoremstyle{definition}

\newtheorem{remark}{Remark}[section]

\numberwithin{equation}{section}
\newcommand{\n}{\mathbb{N}}
\newcommand{\<}{\left\langle}  
\renewcommand {\>}{\right\rangle}  
\newcommand{\norma}[1]{\left\|#1\right\|}

\newcommand{\pr}{\mathbb{P}}

\newcommand{\ew}{\mathbb{E}}

\newcommand{\dd}{\mathbf{d}}

\title{
The central limit theorem for Markov processes that are exponentially ergodic in the bounded-Lipschitz norm}
\author[1]{Dawid Czapla\,%\orcidlink{0000-0002-0562-773X}
}
\author[1]{Katarzyna Horbacz\,%\orcidlink{0000-0003-2900-0031}
}
\author[1,2]{Hanna~Wojew\'odka-\'Sci\k{a}\.zko\,%\orcidlink{0000-0002-8248-5018}
\thanks{Corresponding author; e-mail address: \href{mailto:hanna.wojewodka@gmail.com}{h.wojewodka@iitis.pl}}}
\affil[1]{\small Institute of Mathematics, University of Silesia in Katowice, Bankowa 14, 40-007 Katowice, Poland}
\affil[2]{\small 
Institute of Theoretical and Applied Informatics, Polish Academy of Sciences, Ba\l tycka 5, 44-100 Gliwice, Poland}
\date{}
\begin{document}
\maketitle
\begin{abstract}
In this paper, we establish a version of the central limit theorem for Markov--Feller continuous time processes (with a Polish state space) that are exponentially ergodic in the bounded-Lipschitz distance and enjoy a continuous form of the Foster--Lyapunov condition. As an example, we verify the assumptions of our main result for a specific piecewise-deterministic Markov process, whose deterministic component evolves according to continuous semiflows, switched randomly at the jump times of a~Poisson process.
\end{abstract}

\section*{Introduction}\label{sec:intro}

We are concerned with the~asymptotic behavior of the process $\Gamma_g:=\{t^{-1/2}\hspace{-0.1cm}\int_0^t g\left(\Psi(s)\right)ds\}_{t\geq 0}$, where \hbox{$\Psi:=\{\Psi(t)\}_{t\geq 0}$} is a non-stationary, time-homogeneous Markov process evolving on a~Polish metric space $E$, with an arbitrary transition semigroup $\{P(t)\}_{t\geq 0}$, and $g:E\to\mathbb{R}$ is a bounded Lipschitz continuous observable. More specifically, our main goal is to provide some testable conditions on $\{P(t)\}_{t\geq 0}$ under which $\Psi$ has a unique invariant distribution, say $\mu_*$, and {$\Gamma_{\bar{g}}(t)$}, with $\bar{g}:=g-\int_E g\,d\mu_*$,
converges in law (as $t\to\infty$) to a~centered normal random variable, or, in other words, under which the process $\{\bar{g}(\Psi(t))\}_{t\geq 0}$ obeys the central limit theorem (CLT). 

The CLT is definitely the fundamental one in probability theory and statistics. 
Initially formulated for independent and identically distributed random variables, it was thereafter generalized to martingales (see \cite{levy35}), which has constituted a background for proving various versions of the CLT pertaining to Markov processes. 
First results in this field deal with stationary Markov chains (with discrete time) for which the existence of a $\mu_*$-square integrable solution to the  Poisson equation is guaranteed (see, e.g., \cite{derrenic_lin01, gordin_lifsic78, gordin_lifsic81}). During the later years, many attempts have been made to relax this assumption. For instance, \cite{kipnis_varadhan86} refers to the so-called reversible Markov chains and is based on approximating (in a certain sense) the solutions of the Poisson equation, while \cite{mw} introduces a testable condition relying on the convergence of some series. Another noteworthy article is \cite{subgeom}, where, among others, the principal hypothesis of \cite{mw} is reached by assuming a subgeometric rate of convergence of Markov chain's distribution to a stationary one in terms of the  Wasserstein distance. Furthermore, it should be mentioned that, over the years, the CLT has been also established for certain stationary Markov processes with continuous time parameter; {see e.g., \cite{Bhattacharya, klo} and the results of \cite{holzmann05} for ergodic processes with normal generators (extending those of \cite{gordin_lifsic81}).}

In recent times, however, most attention has been paid to non-stationary Markov processes. Some classical results on the CLT in this case can be found in \cite{mt93}. They involve positive Harris reccurent and aperiodic (discrete-time) Markov chains (or, equivalently, those which are irreducible and ergodic in the total variation norm), for which a drift condition towards petite sets is fulfilled (which guarantees the existence of a suitable solution to the Poisson equation). Such requirements are, however, practically unattainable in non-locally compact state spaces. A version of the CLT for a subclass of non-stationary Markov chains evolving on a general (Polish) metric space, based on a kind of geometric ergodicity in the bounded-Lipschitz distance and the `second order'  Foster-Lyapunov type condition (where a solution to the Poisson equation is not required) is established in \cite{clt}.  In the context of processes with continuous time parameter, probably the most general result of this kind to date, but {relying} on the exponential ergodicity in the Wasserstein distance {(additionally, distinct in nature from that assumed in \cite{clt} or \cite{ergodic_pdmp})}, is stated in \cite{kom_walczuk}. An analogous result in the discrete-time case can be found in \cite{gulgowski}.

The results established in this article are mainly inspired by \cite{kom_walczuk}. A major motivation for the current study was the inability to directly apply the CLT established by Komorowski and Walczuk \cite{kom_walczuk} to some subclass of piecewise-deterministic Markov processes {(PDMPs)},  at least under (relatively natural) conditions imposed in \cite[Proposition 7.2]{ergodic_pdmp} (see also \hbox{\cite{ergodic_prop, dawid-asia}}).

The problem lies in accomplishing the exponential mixing in the sense of condition (H1), employed in \cite{kom_walczuk} (cf. also \cite{gulgowski}), which requires a form of the Lipschitz continuity of each $P(t)$ with respect to the \textit{Wasserstein distance} $d_{\text{W}}$ (see, e.g., \cite[p. 5]{kom_walczuk} for its definition). More precisely, the authors assume the existence of  $\gamma>0$ and $c<\infty$ such that for any two Borel probability measures $\mu$ and $\nu$ (with finite first moments) the following holds:
\begin{equation}\label{cond:kom}
d_{\text{W}}\left(\mu P(t),\nu P(t)\right)\leq ce^{-\gamma t}d_{\text{W}}\left(\mu,\nu\right)\;\;\;\text{for any}\;\;\;t\geq 0.
\end{equation}
We have therefore recognized the need to provide a new, somewhat more useful, criterion that would involve a weaker form of the above requirement, similar to those occuring, for instance, in  \cite{ergodic_prop, ergodic_pdmp, dawid-asia,kapica_sleczka,hania} (cf. also \cite{hairer}). More precisely, instead of \eqref{cond:kom}, we assume that there is $\gamma>0$ such that, for any two Borel probability measures $\mu$ and~$\nu$, there exist a~continuous function $V:E\to [0,\infty)$ and constants $\beta>0$, $\delta\in (0,1)$, for which
\begin{equation}\label{cond:our}
d_{\text{FM}}\left(\mu P(t), \nu P(t)\right) \leq \beta e^{-\gamma t} \left(\int_E V \, d(\mu+\nu)+1\right)^{\delta}  \quad \text{for any} \quad t\geq 0,
\end{equation}
where $d_{\text{FM}}$ stands for the \textit{bounded-Lipschitz distance}, also known as the \textit{Fortet--Mourier metric} (cf., e.g., \cite[p. 236]{lasota_from_fractals} or \cite[p. 192]{bogachev}). 
Besides, an additional advantage of our approach is that this metric is weaker than the Wasserstein one, among others, in a manner enabling the use of a coupling argument (introduced by M. Hairer in \cite{hairer}, and further applied, e.g., in \cite{ergodic_prop, clt, ergodic_pdmp, dawid-asia,kapica_sleczka, sleczka, hania}) to reach an exponential mixing property with respect to $d_{\text{FM}}$, 
which fails while demanding it in terms of $d_{\text{W}}$.

Hypotheses (H2) and (H3) used in \cite{kom_walczuk} roughly ensure that the semigroup $\{P(t)\}_{t\geq 0}$ preserves the  finiteness of measure moments of a given, greater than $2$, order. In the present paper, they are both replaced by a strengthened version of the continuous Lyapunov condition, from which, in practice, such properties are usually derived.

As mentioned above, the proof of our main result, that is, Theorem \ref{thm:main}, is in many places based on the reasoning presented in \cite{kom_walczuk}. Nevertheless, it should be emphasized that without a Lipschitz type assumption on the semigroup $\{P(t)\}_{t\geq 0}$, such as \eqref{cond:kom} (or its discrete-time analogue, employed, e.g., in \cite{gulgowski}), proving the principal limit theorems, like the central one or the law of the iterated logarithm, requires some more subtle arguments, which is reflected, e.g., in \cite{clt,lil} or \cite{kom_szar_peszat}. Most importantly, under condition \eqref{cond:our}, the so-called \textit{corrector function} $\chi:E\to\mathbb{R}$, given by
$$\chi(x)=\int_0^{\infty}P(t)\bar{g}(x)\,dt\;\;\;\text{for any}\;\;\;x\in E,$$
does not need to be Lipschitzian (which is a meaningful argument in the proof of \cite[Theorem 2.1]{kom_walczuk}), but is only continuous.  
Another problem arising in our setting is that the weak convergence of the process distribution (towards the stationary one), guaranteed by \eqref{cond:our}, yields the convergence of the corresponding integrals as long as the integrands are (apart of being continuous) bounded, which is not required while using the Wasserstein distance. This fact prevents, among others, a direct adaptation of the final argument used in the proof of \cite[Lemma 5.5]{kom_walczuk}. We have overcome this obstacle (see Lemma \ref{lem:M2}) by making the use of \cite[Lemma 8.4.3]{bogachev}, which allows replacing the boundedness of the integrand by its uniform integrability with respect to the family of measures constituing the convergent sequence under consideration. 
Finally, let us indicate that the key role in our proof is played by Lemma \ref{lem:E_x_Z(i)^p}. In contrast, results of this nature are not necessary in \cite{kom_walczuk}, since they are in some sense buid-in a priori in hypotheses (H2) and (H3). 

The article is organized as follows. In Section~\ref{sec:1}, 
we gather notation used throughout the paper, as well as recall some basic definitions and facts in the field of measure theory and Markov semigroups. We also quote here a version of the CLT for martingales, crucial for the reasoning line presented in the paper. Section~\ref{sec:2} is devoted to formulating assumptions and the main result, namely Theorem~\ref{thm:main}. {In this part, it is also shown that the~assumptions employed imply the existence of a unique invariant distribution of $\Psi$. The~proof of the~main theorem, along with all auxiliary results, is given in Section~\ref{sec:3}. Further, in Section~\ref{sec:4}, drawing on some ideas from~\cite{Bhattacharya}, we provide a concise representation of the variance of the~limiting normal distribution (involved in Theorem \ref{thm:main}). Additionaly, in Section~\ref{sec:fclt}, we derive a~straightforward conclusion from  \hbox{\cite[Theorem 2.1]{Bhattacharya}} concerning the functional CLT in the stationary case. Finally, in Section~\ref{sec:ex}, we demonstrate the usefulness of the main result by applying it to  establish the CLT for the PDMPs considered in \cite{ergodic_pdmp}.}

\section{Preliminaries}\label{sec:1}
First of all, put $\mathbb{N}_0:=\mathbb{N}\cup\{0\}$ and $\mathbb{R}_+=[0,\infty)$. In what follows, we shall consider a~complete separable metric space $(E,\rho)$, endowed with its Borel $\sigma$-field $\mathcal{B}(E)$. By $\operatorname{B}_b(E)$ we will denote the Banach space of all real-valued, Borel measurable bounded functions on $E$, equipped with the supremum norm $\|\cdot\|_{\infty}$. The subspaces of $\operatorname{B}_b(E)$ consisting of all continuous functions and all Lipschitz continuous functions shall be denoted by $\operatorname{C}_b(E)$ and $\operatorname{Lip}_b(E)$, respectively. Throughout the paper, we will also refer to a particular subset $\operatorname{Lip}_{b,1}(E)$ of $\operatorname{Lip}_b(E)$, defined as 
$$
\operatorname{Lip}_{b,1}(E):=\left\{f\in \operatorname{Lip}_b(E):\;\|f\|_{\text{BL}}\leq 1\right\},
$$
where the norm $\|\cdot\|_{\text{BL}}$ is given by 
$$
\|f\|_{\text{BL}}:=\max\left\{\|f\|_{\infty},\;\;\sup_{x\neq y}\frac{|f(x)-f(y)|}{\rho(x,y)}\right\} \;\;\;\text{for any}\;\;\;f\in\operatorname{Lip}_b(E).
$$

Furthermore, we will write $\mathcal{M}(E)$ for the space of all finite non-negative Borel measures on $E$. The subset of $\mathcal{M}(E)$ consisting of all probability measures will be, in turn, denoted by $\mathcal{M}_1(E)$. Moreover, for any given Borel measurable function $V:E\to\mathbb{R}_+$ and any $r>0$, let us {define} the subset $\mathcal{M}^V_{1,r}(E)$ of $\mathcal{M}_1(E)$ consisting of all measures with finite $r$-th moment with respect to $V$, that is, 
$$
\mathcal{M}^V_{1,r}(E):=\left\{\mu\in\mathcal{M}
_1(E):\;\int_EV^r(x)\mu(dx)<\infty\right\}.
$$
Clearly, for any $s\in (0,2]$, we have $\mathcal{M}^V_{1,2}(E)\subset \mathcal{M}_{1,s}^V(E)$, since $\int_E V^s\,d\mu \leq (\int_E V^2\,d\mu)^{s/2}$ for every $\mu\in\mathcal{M}_1(E)$, due to the H\"older inequality. { Also worth noting here is that, for every $r>0$, $\mathcal{M}^V_{1,r}(E)$ contains all Dirac measures $\delta_x$, $x\in E$, and, when $V$ is continuous, all compactly supported Borel probability measures on $E$.}

For brevity, we will often write $\langle f,\mu\rangle$ for the Lebesgue integral $\int_Ef\,d\mu$ of a Borel measurable function $f:E\to\mathbb{R}$ with respect to a signed Borel measure $\mu$, provided that it exists.

To evaluate the distance between measures, we will use the \emph{Fortet--Mourier metric} (equivalent to
the one induced by the Dudley norm), which on $\mathcal{M}(E)$ is given by
$$
d_{\text{FM},\rho}(\mu,\nu):=\sup_{f\in\operatorname{Lip}_{b,1}(E)}\left|\langle f,\mu-\nu\rangle\right|\;\;\;\text{for any}\;\;\;\mu,\nu\in\mathcal{M}(E).\vspace{-0.2cm}
$$

Let us recall that a sequence $\{\mu_n\}_{n\in\mathbb{N}_0}\subset \mathcal{M}(E)$ of measures is called \emph{weakly convergent} to a measure $\mu\in\mathcal{M}(E)$, which is denoted by $\mu_n\stackrel{w}{\to}\mu$, if for any $f\in \operatorname{C}_b(E)$ we have $\lim_{n\to\infty}\langle f,\mu_n\rangle=\langle f,\mu\rangle$. It is well-known  (see, e.g., \cite[Theorems 8 and 9]{dudley}) that, if $(E,\rho)$ is Polish (which is the case here), then the weak convergence of any sequence of probability measures is equivalent to its convergence in the Fortet--Mourier distance, and also the space $(\mathcal{M}_1(E), d_{\text{FM},\rho})$ is complete. 

In fact, according to \cite[Lemma 8.4.3]{bogachev}, the weak convergence $\mu_n\stackrel{w}{\to}\mu$ of probability measures ensures the convergence of the corresponding integrals even in the case of continuous but not necessarily bounded functions, provided that they are uniformly integrable with respect to $\{\mu_n\}_{n\in\mathbb{N}_0}$. An easily verifiable condition guaranteeing this property, and  thus also the above indicated statement, is presented in the following result: 

\begin{lemma}\label{lem:bogachev}
Let $\{\mu_n\}_{n\in\mathbb{N}_0}\subset \mathcal{M}_1(E)$ be weakly convergent to some $\mu\in\mathcal{M}_1(E)$. Then, for every continuous function $h:E\to\mathbb{R}$ satisfying\, $\sup_{n\in\mathbb{N}_0}
\left\langle |h|^q,\mu_n\right\rangle <\infty$
with some $q>1$, we have\, $\lim_{n\to\infty}\left\langle h,\mu_n\right\rangle=\left\langle h,\mu\right\rangle$.
\end{lemma}
\begin{proof}
Let $h:E\to\mathbb{R}$ be a continuous function such that $M:=\sup_{n\in\mathbb{N}_0}
\left\langle |h|^q,\mu_n\right\rangle <\infty$  with some $q>1$. Then, for every $R>0$, we have
\begin{align*}
\int_{\{|h|\geq R\}} |h(x)|\,\mu_n(dx)
&=\frac{1}{R^{q-1}} \int_{\{|h|\geq R\}} R^{q-1}|h(x)|\,\mu_n(dx)
 \leq \frac{1}{R^{q-1}} \int_{\{|h|\geq R\}} |h(x)|^q\,\mu_n(dx)\\
&\leq  \frac{1}{R^{q-1}}  \left\langle |h|^q,\mu_n\right\rangle \leq \frac{M}{R^{q-1}}
\quad \text{for all}\quad n\in\mathbb{N}_0.
\end{align*}
Keeping in mind that $q>1$, we therefore see that
$$
\lim_{R\to \infty} \sup_{n\in\mathbb{N}_0}\, \int_{\{|h|\geq R\}} |h(x)|\,\mu_n(dx)=0,
$$
but this, according to \cite[Lemma 8.4.3]{bogachev}, already implies the assertion of the lemma.
\end{proof}

\subsection{{Markov semigroups}}\label{sec:markov_operators}
Let us now recall some basic concepts from the theory of Markov operators, to be employed in the remainder of the paper.

A function $P : E \times \mathcal{B}(E) \to [0, 1]$ is called a \emph{stochastic kernel} if $E\ni x\mapsto P(x, A)$ is a~Borel measurable map for each $A\in\mathcal{B}(E)$, and $\mathcal{B}(E)\ni A \mapsto P(x, A)$ is a Borel probability measure for each $x\in E$. The composition of any two such kernels, say $P$ and $Q$, is defined by
\begin{equation}\label{composition}
PQ(x,A)=\int_EQ(y,A)\,P(x,dy)\;\;\;\text{for any}\;\;\;x\in E\;\;\;\text{and any}\;\;\;A\in\mathcal{B}(E).
\end{equation}

Given a stochastic kernel $P$, we can define two corresponding operators (here denoted by the same symbol, according to the convention employed, e.g., in \cite{mt93,klo}); one acting on~$\mathcal{M}(E)$, defined by 
\begin{equation}\label{def:muP}
\mu P(A):=\int_E P(x,A)\,\mu (dx)\;\;\;\text{for}\;\;\;\mu\in\mathcal{M}(E)\;\;\;\text{and}\;\;\;A\in\mathcal{B}(E),
\end{equation}
and the second one acting on $\operatorname{B}_b(E)$, given by
\begin{equation}\label{def:Pf}
Pf(x):=\int_Ef(y)\, P(x,dy)\;\;\;\text{for}\;\;\;f\in \operatorname{B}_b(E)\;\;\;\text{and}\;\;\;x\in E.
\end{equation}
These operators are related to each other so that
\begin{equation}\label{duality_property}
\left\langle f,\mu P\right\rangle=\left\langle Pf,\mu \right\rangle\;\;\;\text{for any}\;\;\;f\in \operatorname{B}_b(E)\;\;\;\text{and any}\;\;\;\mu\in\mathcal{M}(E).
\end{equation}
The operator $(\cdot)P : \mathcal{M}(E)\to\mathcal{M}(E)$, given by \eqref{def:muP}, is called a \emph{(regular) Markov operator}, and $P(\cdot): \operatorname{B}_b(E) \to \operatorname{B}_b(E)$, defined by \eqref{def:Pf}, is said to be its \emph{dual operator}. Obviously, the Markov operator (resp. its dual) corresponding to the composition of two given kernels in the sense of~\eqref{composition} is just the usual composition of the Markov operators (resp. their duals) induced by these kernels. Let us also highlight that the right-hand sides of \eqref{def:Pf} and \eqref{duality_property}, in fact, make sense for all Borel measurable, bounded below functions, and we will often write $Pf$ also for such functions~$f$ (obviously, in that case, $Pf$ takes values in $\mathbb{R}\cup\{\infty\}$).

Let $\mathbb{T}\in\{\mathbb{R}_+,\mathbb{N}_0\}$. A family of stochastic kernels $\{P(t)\}_{t\in\mathbb{T}}$ 
(or the induced family of Markov operators) is called a \emph{(regular) Markov semigroup} whenever \vspace{-0.1cm}
$$
P(s)P(t) = P(s+t)\quad\text{for any}\quad s,t\in\mathbb{T},\quad P(0)(x,\cdot) = \delta_x \quad\text{for every}\quad x\in E,\vspace{-0.1cm}
$$ 
and, if $T=\mathbb{R}_+$, the map \hbox{$\mathbb{R}_+\times E\ni (t,x)\mapsto P(t)(x,A)$} is $\mathcal{B}(\mathbb{R}_+\times E)/\mathcal{B}(\mathbb{R})$-measurable for any $A\in\mathcal{B}(E)$ (cf. the definition  of \emph{transition function} in \cite[p. 156]{ethier}).

We call a Markov semigroup $\{P(t)\}_{t\in\mathbb{T}}$  \emph{Feller} if, for each $t\in \mathbb{T}$, the dual operator~$P(t)(\cdot)$ preserves continuity of bounded functions, i.e., $P(t)(\operatorname{C}_b(E))\subset \operatorname{C}_b(E)$. A~measure \hbox{$\mu_*\in\mathcal{M}(E)$} is said to be \emph{invariant} for such a semigroup $\{P(t)\}_{t\in\mathbb{T}}$ whenever \hbox{${\mu}_*P(t)={\mu}_*$} for every $t\in\mathbb{T}$. Obviously, these concepts can be also referred to a single Markov operator.

Given a Markov semigroup $\{P(t)\}_{t\in\mathbb{T}}$ of stochastic kernels on $E\times \mathcal{B}(E)$, by a \emph{time-homogeneous Markov process} with \emph{transition semigroup} $\{P(t)\}_{t\in\mathbb{T}}$ we mean a family of $E$-valued random variables \hbox{$\Psi:=\{\Psi(t)\}_{t\in\mathbb{T}}$} 
on some probability space $(\Omega,\mathcal{F},\mathbb{P})$ such that, for any $s, t\in\mathbb{T}$, $A \in\mathcal{B}(E)$, and $x\in E$,
\begin{gather}\label{P_mu_cnt}
\begin{aligned}
&\mathbb{P}\left(\Psi(s+t)\in A\,|\,\mathcal{F}(s)\right)=\mathbb{P}\left(\Psi(s+t)\in A\,|\,\Psi(s)\right)\;\;\text{a.s.},\\
&\pr(\Psi(s+t)\in A\,|\,\Psi(s)=x)=P(t)(x,A),
\end{aligned}
\end{gather}
where $\{\mathcal{F}(s)\}_{s\in\mathbb{T}}$ is the natural filtration of $\Psi$. {Obviously, $\Psi$ can also be regarded as a~process on the probability space $(\Omega,\mathcal{F}_{\infty},\mathbb{P}|_{\mathcal{F}_{\infty}})$ with $\mathcal{F}_{\infty}:=\sigma\left(\{\Psi(t):\,t\in\mathbb{T}\}\right)$. The distribution of $\Psi(0)$ is referred to as the initial one.} If $\mathbb{T}=\mathbb{N}_0$ and $P(n)=P(1)^n$ (i.e., $P(n)$ is the $n$th iteration of $P(1)$) for every $n\in\mathbb{N}$, then $\Psi$ satisfying \eqref{P_mu_cnt} is usually called a \emph{Markov chain} (rather than a~process), and the kernel $P(1)$ is then said to be the \emph{one-step transition law} of this chain. Let us also note that, letting $\mu(t)$ be the distribution of $\Psi(t)$ for every~$t\in\mathbb{T}$, we have
$\mu(s+t) = \mu(s)P(t)$ for any $s, t\in\mathbb{T}$, and thus it is indeed reasonable to call $\{P(t)\}_{t\in\mathbb{T}}$ a transition semigroup.

{
Let us also recall that a continuous-time stochastic process $\Psi=\{\Psi(t)\}_{t\in\mathbb{R}_+}$, adapted to a filtration $\{\mathcal{F}(t)\}_{t\in\mathbb{R}_+}$, is called \emph{jointly} (resp. \emph{progressively}) measurable if the map \hbox{$\mathbb{R}_+\times \Omega \ni (t,\omega)\mapsto \Psi(t)(\omega)\in E$} (resp. its restriction to $[0,t]\times \Omega$) is \hbox{$\mathcal{B}(\mathbb{R}_+)\otimes \mathcal{F}_{\infty} / \mathcal{B}(E)$}-measurable (resp. \hbox{$\mathcal{B}([0,t])\otimes \mathcal{F}(t) / \mathcal{B}(E)$}-measurable for every $t>0$). As~is well known, every adapted process with right- or left-continuous sample paths is progressively measurable, and thus also jointly measurable.}

Throughout the paper, { for a~given Markov process~$\Psi=\{\Psi(t)\}_{t\in\mathbb{T}}$}, we shall use the so-called \emph{Dynkin set-up} (see, e.g., \cite[p. 7]{sharpe}), which brings a family $\{\pr_x:\,x\in E\}$ of probability measures on~$\mathcal{F}$ such that, for every $x\in E$, one has $\pr_x(\Psi(0)=x)=1$ and \eqref{P_mu_cnt} is fulfilled with $\pr_x$ in place of~$\pr$. Obviously,~$\pr_x$ might be then thought of as the conditional distribution of $\pr$, given the initial state $x$ of $\Psi$, i.e., $\pr_x(\cdot):=\pr(\cdot\,|\,\Psi(0)=x)$. Within this framework, for each $\mu\in\mathcal{M}_1(E)$, one can define 
\begin{equation}
\label{e:pr_m}
\pr_{\mu}(F):=\int_E \pr_x(F)\,\mu(dx)\quad\text{for any}\quad F\in\mathcal{F},
\end{equation}
and check that $\Psi$ has the Markov property in the sense of \eqref{P_mu_cnt} relative to $\pr_{\mu}$, with the given transition semigroup and initial law $\mu$. For any $\mu\in\mathcal{M}_1(E)$ and $x\in E$, the expectation operators w.r.t. $\pr_{\mu}$ and $\pr_x (=\pr_{\delta_x})$ will be denoted by $\ew_{\mu}$ and $\ew_x$, respectively. It is easy to see that, given any Borel measurable, bounded below function $f:E\to\mathbb{R}$, we then have
\begin{equation}\label{eq:Ef}
P(t)f(x)=\ew_x\left(f(\Psi(t))\right)\quad\text{for all}\quad x\in E,\,t\in \mathbb{T},
\end{equation}
and thus it follows from \eqref{P_mu_cnt} that 
$$\ew_{\mu}\left(f(\Psi(s+t))\,|\,\mathcal{F}(s)\right)=P(t)f(\Psi (s))=\ew_{\Psi(s)}\left(f(\Psi(t))\right)\quad\text{for all}\quad \mu\in\mathcal{M}_1(E),\;s,t\in\mathbb{T}.\vspace{-0.05cm}$$

In practice, it will be convenient to work in the following canonical setting. Given a~Markov process $\Psi$, we put $\widetilde{\Omega}:=\{\Psi(\cdot)(\omega):\; \omega\in\Omega\}\subset E^{\mathbb{T}}$, and, for every $t\in\mathbb{T}$, we define the projection $\widetilde{\Psi}(t):\widetilde{\Omega}\to E$ by $\widetilde{\Psi}(t)(\tilde{\omega}):=\tilde{\omega}(t)$\, for any\, $\tilde{\omega}\in\widetilde{\Omega}$, as well as the $\sigma$-fields\vspace{-0.05cm}
$$\widetilde{\mathcal{F}}(t):=\sigma\left(\left\{\widetilde{\Psi}(s):\,s\in [0,t]\cap\mathbb{T}\right\}\right)\quad\text{and}\quad \mathcal{\widetilde{F}_{\infty}}:=\sigma\left(\left\{\widetilde{\Psi}(s):\,s\in\mathbb{T}\right\}\right). \vspace{-0.2cm}$$
Further, we introduce the map \hbox{$\pi:\Omega\to \widetilde{\Omega}$} given by $\pi(\omega)(t):=\Psi(t)(\omega)$ for all $\omega\in\Omega$ and~\hbox{$t\in\mathbb{T}$}. Then $\widetilde{\Psi}(t)\circ\pi=\Psi(t)$ for every $t\in\mathbb{T}$ and it follows easily that $\pi$ is both $\mathcal{F}_{\infty}/\mathcal{\widetilde{F}_{\infty}}$ and $\mathcal{F}(t)/\mathcal{\widetilde{F}}(t)$-measurable for each $t\in\mathbb{T}$. Now, let us define $\widetilde{\pr}_{\mu}(\widetilde{F}):=\pr_{\mu}(\pi^{-1}(\widetilde{F}))$ for any $\widetilde{F}\in\widetilde{\mathcal{F}}_{\infty}$ and $\mu\in\mathcal{M}_1(E)$. Then, for each $\mu\in\mathcal{M}_1(E)$,  $\widetilde{\pr}_{\mu}$ is a probability measure on $\widetilde{\mathcal{F}}_{\infty}$, and the processes $\widetilde{\Psi}$ and $\Psi$ have the same finite dimensional distributions under $\pr_{\mu}$ and $\widetilde{\pr}_{\mu}$, respectively, i.e.,
$$\widetilde{\pr}_{\mu}\left(\widetilde{\Psi}(t_0)\in A_0,\ldots,\widetilde{\Psi}(t_n)\in A_n\right)=\pr_{\mu}\left(\Psi(t_0)\in A_0,\ldots,\Psi(t_n)\in A_n\right) \vspace{-0.2cm}$$
for all $n\in\n_0$, $t_0\leq t_1\leq\ldots\leq t_n$ in $\mathbb{T}$, and $A_0,\ldots,A_n\in\mathcal{B}(E)$. In particular,  $\widetilde{\Psi}$ is then a~Markov process w.r.t. $\{\widetilde{\mathcal{F}}(t)\}_{t\in\mathbb{T}}$ with the transition semigroup such as that of $\Psi$  (cf.~the~proof of \cite[Theorem 4.3]{blumenthal}). { Moreover, it is not hard to check that, if $\Psi$ (with $\mathbb{T}=\mathbb{R}_+$) is jointly (resp.~progressively) measurable, then $\widetilde{\Psi}$ is jointly (resp. progressively) measurable as well.} One important advantage of this canonical setting is that we can consider the shift operators $\Theta_t:\widetilde{\Omega}\to \widetilde{\Omega} $, $t\in\mathbb{T}$, defined by
\begin{equation}   
\label{e:shift}
\Theta_t(\tilde{\omega})(s):=\tilde{\omega}(s+t)\quad\text{for}\quad \tilde{\omega}\in\widetilde{\Omega},\;s\in\mathbb{T},
\end{equation}
which allows one to write $\widetilde{\Psi}(s)\circ\Theta_t=\widetilde{\Psi}(s+t)$ for any $s,t\in\mathbb{T}$. This, among others, enables the use of the Birkhoff ergodic theorem in terms of the Markov process under consideration.

\subsection{A version of the {CLT} for martingales}

While proving the main result of this article, we will refer to \cite[Theorem 5.1]{kom_walczuk}, which is quoted below for the convenience of the reader.

Let $(\Omega,\{\mathcal{F}_n\}_{n\in\mathbb{N}_0},\mathcal{F},\mathbb{P})$
be a filtrated probability space with trivial $\mathcal{F}_0$, and consider 
a~square integrable martingale $\{m_n\}_{n\in\mathbb{N}_0}$, as well as 
the sequence $\{z_n\}_{n\in\mathbb{N}}$ of its increments, given by $z_n=m_n-m_{n-1}$ for $n\in\mathbb{N}$. Further, define $\langle m\rangle_n$, $n\in\mathbb{N}$, as
\begin{equation}\label{def:<m>}
\langle m\rangle_n:=\sum_{i=1}^n\mathbb{E}\left(z_i^2|\mathcal{F}_{i-1}\right).
\end{equation}

\begin{theorem}[{\cite[Theorem 5.1]{kom_walczuk}}]\label{thm:martingale}
Suppose that the following conditions hold:
\begin{enumerate}[label=\textnormal{(M\arabic*)}]
\item\label{cnd:m1}
For every $\varepsilon>0$ we have
\begin{equation*}
\lim_{n\to\infty}\frac{1}{n}\sum_{i=0}^{n-1}\mathbb{E}\left(
z_{i+1}^2 
\mathbbm{1}_{\left\{\left|z_{i+1}\right|\geq\varepsilon\sqrt{n}\right\}}\right)=0
\end{equation*}
\item\label{cnd:m2}
We have\; $\sup_{n\geq 1}\mathbb{E}\left(z_n^2\right)<\infty$
and there exists { $\sigma\in [0,\infty)$} such that
\begin{equation*}
\lim_{k\to\infty}\limsup_{l\to\infty}\frac{1}{l}\sum_{j=1}^l\mathbb{E}\left|\frac{1}{k}\mathbb{E}\left(\langle m\rangle_{jk}-\langle m\rangle_{(j-1)k}|\mathcal{F}_{(j-1)k}\right)-\sigma^2\right|=0.
\end{equation*}
\item\label{cnd:m3}
For every $\varepsilon>0$ we have
\begin{equation*}
\lim_{k\to\infty}\limsup_{l\to\infty}\frac{1}{kl}\sum_{j=1}^l\sum_{i=(j-1)k}^{jk-1}\mathbb{E}
\left(
\left(1+z_{i+1}^2\right)
\mathbbm{1}_{\left\{\left|m_i-m_{(j-1)}k\right|\geq\varepsilon\sqrt{kl}\right\}}
\right)
=0.
\end{equation*}
\end{enumerate}
Then
\begin{equation}\label{eq:lim_sigma}
\lim_{n\to\infty}\frac{\mathbb{E}\langle m\rangle_n}{n}=\sigma^2,
\end{equation}
and $\{m_n\}_{n\in\mathbb{N}_0}$ obeys the CLT, i.e.,
$$
\lim_{n\to\infty}\mathbb{P}\left(\frac{1}{\sqrt{n}}\sum_{i=0}^{n-1}m_i\leq u\right)=\Phi_{\sigma}(u),
$$
where $\Phi_{\sigma}$ is the distribution function of a centered normal law with variance $\sigma^2$.
\end{theorem}
Obviously, the centered normal distribution with zero variance (i.e., $\sigma^2=0$) is viewed as the Dirac measure at $0$.

\begin{remark}
Basically, almost all known central convergence criteria for martingales rely on \cite[Theorem 2]{brown71} of Brown, which, in turn, is based on the celebrated Lindeberg condition. In~the result above, this condition is stated as hypothesis \ref{cnd:m1} and is also reflected in assumption \ref{cnd:m3}. The Lindeberg condition originates  from \hbox{\cite[Theorem 27.2]{billingsley86}} (see also~\hbox{\cite[pp. 292--294]{loeve77}}), that is, the Lindeberg-Feller CLT, which concerns independent, but not necessarily identically distributed, random variables.
The Brown result, apart from this condition, involves the assumption that
\begin{equation}\label{*}
\frac{\langle m\rangle_n}{\mathbb{E}\langle m\rangle_n}\to 1\quad \text{in probability}.
\end{equation}
This requirement is, however, relatively hard to verify in practice, and thus it is often replaced 
with other assumptions. For instance, the martingale CLT given in \hbox{\cite[Theorem D.6.4]{mt93}} requires instead that
\begin{equation}\label{**}
\lim_{n \to \infty} \frac{1}{n} \sum_{i=1}^n \mathbb{E}\left(z_i^2 | \mathcal{F}_{i-1}\right) = \sigma^2 \quad\text{a.s.}\quad\text{for some $\sigma^2 \in (0,\infty)$,}
\end{equation}
which can be simply written as $\lim_{n \to \infty} \langle m\rangle_n/n 
= \sigma^2$ a.s. This assumption, together with the Lindeberg condition, enables one to derive \eqref{eq:lim_sigma}, which then yields \eqref{*}. Hypotheses \ref{cnd:m2} and \ref{cnd:m3}, adapted here from \cite{kom_walczuk}, although more technical, prove to be less restrictive than \eqref{**}. As emphasized by the authors, \cite[Theorem 5.1]{kom_walczuk} was inspired by the proof of \cite[Theorem 2.1]{klo}, concerning martingales with stationary increments, where the only assumption is $\lim_{n \to\infty} \mathbb{E}|\langle m\rangle_n/n -\sigma^2| \to 0$. Conditions \ref{cnd:m1}-\ref{cnd:m3} therefore constitute, in some way, a substitute of this assumption and the stationarity, expressed in the spirit of Lindeberg-type conditions and the aforementioned hypothesis \eqref{**} from \cite{mt93}.
\end{remark}

\section{Assumptions and formulation of the main result}\label{sec:2}

Let $\Psi=\{\Psi(t)\}_{t\in\mathbb{R}_+}$ be a jointly measurable, $E$-valued time-homogeneous Markov process with transition semigroup $\{P(t)\}_{t\in\mathbb{R}_+}$. We assume that the process is given in the Dynkin setup with a suitable family \hbox{\{$\pr_{\mu}:\, \mu\in\mathcal{M}_1(E)\}$} of probability measures. Furthermore, for \hbox{analysis} purposes, we will identify~$\Psi$ with the canonical (and also jointly measurable) process~$\widetilde{\Psi}$, defined in Section~\ref{sec:markov_operators}, simultaneously, dropping all the tildes used in the definition of the latter.

To state the main result of this paper, we need to employ several conditions regarding the semigroup $\{P(t)\}_{t\in\mathbb{R}_+}$. Firstly, we assume that
\begin{enumerate}[label=\textnormal{(A\arabic*)}]
\item\label{(h0)} $\{P(t)\}_{t\in\mathbb{R}_+}$ has the Feller property;
\end{enumerate}
and, secondly, we require the existence of a continuous function $V:E \to \mathbb{R}_+$ such that the following holds:
\begin{enumerate}[label=\textnormal{(A\arabic*)}, start=2]
\item\label{(h1)} $\{P(t)\}_{t\in\mathbb{R}_+}$ is \emph{\hbox{$V$-exponentially} mixing in the metric $d_{\text{FM},\rho}$} in the sense that there exist constants $\gamma>0$ and $\beta>0$ such that
$$
d_{\text{FM},\rho} (\mu P(t), \nu P(t))\leq \beta(\<V,\mu\>+\<V,\nu\>+1)^{1/2}e^{-\gamma t}\quad\text{for all}\quad \mu,\nu\in\mathcal{M}_{1,1}^V(E),\;t\in\mathbb{R}_+.
$$
\item\label{(h2)} there exist $A,B\geq 0$ and $\Gamma>0$ such that
$$
P(t)V^2(x)\leq A e^{-\Gamma t}V^2(x)+B\;\;\;\text{for all}\;\;\;x\in E,\;t\in\mathbb{R}_+.
$$
\end{enumerate}

\begin{remark}\label{rem:1}
Assumption \ref{(h2)} is a strengthened form of the Lyapunov condition (see, e.g., \hbox{\cite[Definition 2.1]{cloez_hairer}}) and, among others, it ensures that the semigroup $\{P(t)\}_{t\in\mathbb{R}_+}$ leaves the set $\mathcal{M}_{1,2}^V(E)$ invariant. We can say even more, namely, for any $\mu\in \mathcal{M}_{1,2}^V(E)$ we have\, $\sup_{t\in\mathbb{R}_+} \< V^2,\mu P(t)\><\infty$, since
$$\< V^2,\mu P(t)\>=\< P(t)V^2,\mu\> 
\leq A\left\langle V^2,\mu\right\rangle +B<\infty\quad\text{for all}\quad t\in\mathbb{R}_+.$$  
In Section \ref{sec:aux_facts}, we will derive (in Lemma \ref{lem:P(t)C}) a variant of \ref{(h2)} concerning $P(t)\mathcal{C}^p$ for $p\in (0,4]$ and $\mathcal{C}=\kappa(V+1)^{1/2}$ (with some constant $\kappa$), which leads to a conclusion analogous to that above (see Corollary \ref{lem:sup_t}). These observations will be essential for proving Lemma~\ref{lem:E_x_Z(i)^p}, which, in turn, plays a key role in verifying the hypotheses of Theorem~\ref{thm:martingale} for a suitable martingale (defined in Section \ref{sec:df_mart}).
\end{remark}

We will now show that the conjunction of the above-stated conditions implies that the semigroup $\{P(t)\}_{t\in\mathbb{R}_+}$ is, in fact, \emph{$V$-exponentially ergodic in} $d_{\text{FM},\rho}$.

\begin{lemma}\label{lem:inv_exist}
If conditions \ref{(h0)}-\ref{(h2)} hold with some continuous function $V:E \to \mathbb{R}_+$, then $\{P(t)\}_{t\in\mathbb{R}_+}$ possesses a unique invariant probability measure $\mu_*$, { and $\mu_*\in\mathcal{M}_{1,2}^V(E)$}. Moreover, if such a measure exists, condition \ref{(h1)} is equivalent to the following one: there exist $\gamma>0$ and $\kappa>0$ such that
\begin{equation}\label{dc0}
d_{\text{FM},\rho} (\mu P(t), \mu_*)\leq \kappa(\<V,\mu\>+1)^{1/2}e^{-\gamma t}\quad\text{for every}\quad \mu\in\mathcal{M}_{1,1}^V(E).
\end{equation}
\end{lemma}
\begin{proof} Let $x_0\in E$. Then from hypotheses \ref{(h1)}, \ref{(h2)} and the inequality $PV\leq (PV^2)^{1/2}$ it follows that, for any $s,t\geq 0$,
\begin{equation}\label{e:dc1}
d_{\text{FM},\rho} \left(\delta_{x_0} P(t), \delta_{x_0} P(t+s)\right)\leq \beta\left(V(x_0)+P(s)V(x_0)+1\right)^{1/2}e^{-\gamma t}\leq  W\left(x_0\right)e^{-\gamma t}, 
\end{equation}
where $W(x_0):=\beta(V(x_0)+(AV^2(x_0)+B)^{1/2}+1)^{1/2}$. Hence, for every $t\geq 0$ and any $n,m\in\mathbb{N}$, we have
$$d_{\text{FM},\rho} \left(\delta_{x_0} P(t+n), \delta_{x_0} P(t+n+m)\right)\leq W\left(x_0\right)e^{-\gamma (t+n)}.$$
This shows that $\{\delta_{x_0}P(t+n)\}_{n\in\n}$ is a Cauchy sequence w.r.t. $d_{\text{FM},\rho}$ for every $t\geq 0$. Consequently, since the space $(\mathcal{M}_1(E), d_{\text{FM},\rho})$ is complete, each of such sequences is convergent in this space. On the other hand, \eqref{e:dc1} also implies that
$$d_{\text{FM},\rho} \left(\delta_{x_0} P(s+n), \delta_{x_0} P(t+n)\right)\leq W(x_0)e^{-\gamma ((s\wedge t)+n)} \quad \text{for any}\quad s,t\geq 0,\; n\in\n,$$
which, in turn, guarantees that, all the sequences $\{\delta_{x_0}P(t+n)\}_{n\in\n}$, $t\geq 0$, have the same limit, say $\mu_*\in\mathcal{M}_1(E)$. Obviously, this is equivalent to that $\delta_{x_0}P(t+n)\stackrel{w}{\to} \mu_*$ for all $t\geq 0$. 

Now, using \ref{(h0)} we can conclude that $\mu_*$ is invariant for $\{P(t)\}_{t\in\mathbb{R}_+}$. Indeed, for any $t\geq 0$ and $f\in \operatorname{C}_b(E)$, we get
$$
\left\langle f,\mu_* P(t)\right\rangle=\left\langle P(t)f, \mu_*\right\rangle=\lim_{n\to\infty} \left\langle P(t)f,\, \delta_{x_0}P(n)\right\rangle
=\lim_{n\to\infty} \left\langle f, \delta_{x_0}P(t+n)\right\rangle=\left\langle f,\mu_*\right\rangle,
$$
which proves that $\mu_* P(t)=\mu_*$ for all $t\geq 0$. Moreover, \ref{(h1)} yields that $\delta_x P(t)\stackrel{w}{\to} \mu_*$, as $t\to\infty$, for every $x\in E$, and thus, using Lebesgue's dominated convergence theorem, one can deduce that, in fact, $\mu P(t)\stackrel{w}{\to} \mu_*$ for every $\mu\in\mathcal{M}_1(X)$. This implies that $\mu_*$ has to be the unique invariant probability measure of $\{P(t)\}_{t\in\mathbb{R}_+}$.

Further, observe that $\mu_*\in \mathcal{M}_{1,2}^V(E)$. To see this, let $V_k:=\min\{V^2,k\}$ for $k\in\n$. Since $\delta_{x_0}P(t)\stackrel{w}{\to}\mu_*$ as $t\to\infty$ and $\{V_k\}_{k\in\n}\subset \operatorname{C}_b(E)$, applying \ref{(h2)}, we get
$$\<V_k, \mu_*\>=\lim_{t\to\infty}\<V_k, \delta_{x_0} P(t)\>=\lim_{t\to\infty} P(t)V_k(x_0)\leq B\quad\text{for all}\quad k\in\n.$$
Obviously $V_k(x)\uparrow V^2(x)$ for any $x\in E$. Hence, we can use the Lebesgue monotone convergence theorem to conclude that $\<V^2, \mu_*\>=\lim_{k\to\infty} \<V_k, \mu_*\>\leq B<\infty$, which is the desired claim.

To prove the second statement of the lemma, let $\mu_*\in\mathcal{M}_{1,2}^V(E)$ be an invariant measure of $\{P(t)\}_{t\in\mathbb{R}_+}$. Then \ref{(h1)} implies that, for every $\mu\in\mathcal{M}_{1,1}^V(E)$,
\begin{align*}
d_{\text{FM},\rho} (\mu P(t), \mu_*)&=d_{\text{FM},\rho} (\mu P(t), \mu_* P(t))\leq \beta(\<V,\mu_*\>+\<V,\mu\>+1)^{1/2}e^{-\gamma t}\\
&\leq \beta(\<V,\mu_*\>+1)^{1/2}(\<V,\mu\>+1)^{1/2}e^{-\gamma t},
\end{align*}
whence \eqref{dc0} holds with $\kappa:=\beta(\<V,\mu_*\>+1)^{1/2}.$
Conversely, if \eqref{dc0} is fulfilled, then
\begin{align*}
d_{\text{FM},\rho}& (\mu P(t), \nu P(t))\leq d_{\text{FM},\rho} (\mu P(t), \mu_*)+d_{\text{FM},\rho} (\mu_*, \nu P(t))\\
&\leq \kappa \left((\<V,\mu\>+1)^{1/2}+(\<V,\nu\>+1)^{1/2}\right)e^{-\gamma t}
\leq 2\kappa (\<V,\mu\>+\<V,\nu\>+1)^{1/2}e^{-\gamma t},
\end{align*}
for any $\mu,\nu\in\mathcal{M}_{1,1}^V(E)$, which means that \ref{(h1)} holds with $\beta:=2\kappa$.
\end{proof}

Throughout the rest of the paper, upon assuming \ref{(h0)}-\ref{(h2)},  the unique invariant probability measure of  $\{P(t)\}_{t\in\mathbb{R}_+}$ (which belongs to $\mathcal{M}_{1,2}^V(E)$) will be denoted by $\mu_*$. Moreover, we define
\begin{equation}\label{eq:bar_C}
\mathcal{C}(x):=\kappa(V(x)+1)^{1/2} \quad\text{for}\quad x\in E,
\end{equation}
where $\kappa$ is the constant featured in \eqref{dc0}. Then, Lemma \ref{lem:inv_exist} yields that
\begin{equation}
\label{e:erg_dfm}
d_{\text{FM},\rho} (\delta_x P(t), \mu_*)\leq \mathcal{C}(x)e^{-\gamma t}\quad\text{for every}\quad x\in E.
\end{equation}

The main result of this paper reads as follows:
\begin{theorem}\label{thm:main}
Suppose that the transition semigroup $\{P(t)\}_{t\in\mathbb{R}_+}$ of $\Psi$ satisfies \hbox{hypotheses} \ref{(h0)}-\ref{(h2)} with some continuous function $V:E\to\mathbb{R}_+$. {Then it possesses a unique invariant probability measure $\mu_*$, which belongs to $\mathcal{M}_{1,2}^V(E)$, and, for every $g\in \operatorname{Lip}_b(E)$}, the CLT holds for the process $\{\bar{g}(\Psi(t))\}_{t\in\mathbb{R}_+}$ with \hbox{$\bar{g}:=g-\langle g,\mu_*\rangle$}, {independently of the initial distribution $\mu\in\mathcal{M}_1(E)$ of $\Psi$}, that is,  
\begin{equation}
\label{e:main}
\lim_{t\to\infty}\mathbb{P}_{\mu}\left(\frac{1}{\sqrt{t}}\int_0^t\bar{g}\left(\Psi(s)\right)\,ds\leq u\right)=\Phi_{\sigma}(u)\;\;\;\text{for all}\;\;\;u\in \mathbb{R},
\end{equation}
where $\Phi_{\sigma}$ is the distribution function of a centered normal law with the variance $\sigma^2<\infty$ of the form
\begin{equation}\label{def:sigma}
\sigma^2=\mathbb{E}_{\mu_*}
\left(
\left(
\int_0^{\infty}P(t)\bar{g}(\Psi(1))\,dt-\int_0^{\infty}P(t)\bar{g}(\Psi(0))\,dt+\int_0^1\bar{g}(\Psi(s))\,ds
\right)^2
\right).
\end{equation}

\end{theorem}

\begin{remark}
In fact, our main result remains valid (with almost the same proof, except for some obvious minor changes) under slightly weaker conditions than~\ref{(h1)} and \ref{(h2)}.

Firstly, the exponent $1/2$ on the right-hand of the inequality in \ref{(h1)} can be replaced by any other number $\delta\in (0,1)$. Then \eqref{dc0} holds with $\delta$ in place of $1/2$,  and thus one may consider $\mathcal{C}:=\kappa(V+1)^{\delta}$ instead of \eqref{eq:bar_C}. Upcoming Lemma \ref{lem:P(t)C}, Corollary \ref{lem:sup_t} and Lemma~\ref{lem:E_x_Z(i)^p} can be then established for $p\in (0, 2\delta^{-1}]$ (rather that $p\in (0, 4]$), which is sufficient to prove the main result.

Secondly, \ref{(h2)} can be somewhat relaxed by assuming instead that
\begin{equation}\label{e:a3_gen}
P(t)V^2(x)\leq A e^{-\Gamma t}V^r(x)+B\;\;\;\text{for all}\;\;\;x\in E,\;t\in\mathbb{R}_+
\end{equation}
{ with some $r\geq 2$ (instead of $r=2$). Clearly,} for $r\in(0,2)$, condition \eqref{e:a3_gen} implies \ref{(h2)} (with $2A$ and $2A+B$ in places of $A$ and  $B$, respectively). Moreover, considering \eqref{e:a3_gen} with $t=0$ shows that $V$ has to be bounded in this case. Hence, if $r<2$, then every probability measure has finite all moments w.r.t. $V$, and thus the assertion of Corollary~\ref{lem:sup_t} is trivially satisfied. Nevertheless, since $V$ usually occurs in the form \hbox{$V=\rho(\cdot,x_*)$} (with an~arbitrary $x_*\in E$), such a case is essentially tantamount to assuming the boundedness of the state space, which is a very restrictive requirement.

Regardless of the above, we will stick with the initial version of the assumptions, with $\delta=1/2$ and $r=2$, to make the proofs easier to follow.
\end{remark}

\section{Proof of the main result}\label{sec:3}

\subsection{Some auxiliary facts}\label{sec:aux_facts}

First of all, let us note that, due to the boundedness of $g$, it suffices to prove the CLT for $\{\bar{g}(\Psi(t))\}_{t\in\mathbb{R}_+}$ along the integers.

\begin{lemma}\label{lem:suffices}
Let $\mu\in\mathcal{M}_1(E)$. Then, for any $g\in \operatorname{B}_b(E)$ and
$$I(t):=\int_0^t g(\Psi(s))\,ds,\quad t\geq 0,$$
the process $\{I(t)/\sqrt{t}\}_{t>0}$ converges in law (to some $\mathbb{R}$-valued random variable) under~$\pr_{\mu}$ \hbox{(as $t\to \infty$)} whenever the sequence $\{I(n)/\sqrt{n}\}_{n\in\n}$ does. 
\end{lemma}
\begin{proof}
Fix an arbitrary $\varepsilon>0$ and choose $n_0\in \n$ such that $n_0^{-1/2}<\varepsilon/(2\|g\|_{\infty})$. Then, taking into account the boundedness of $g$, for any $n\geq n_0$ and any $t\in[n,n+1)$, we get
\begin{align*}
\left|\frac{I(t)}{\sqrt{t}}-\frac{I(n)}{\sqrt{n}}\right|
&\leq \frac{\left|I(t)-I(n)\right|}{\sqrt{n}}+\left|I(n)\right|\left(\frac{1}{\sqrt{n}}-\frac{1}{\sqrt{t}}\right)\leq \frac{\|g\|_{\infty}}{\sqrt{n}}+\left|I(n)\right|\left(\frac{1}{\sqrt{n}}-\frac{1}{\sqrt{n+1}}\right)\\
&<\frac{\varepsilon}{2}+\left|I(n)\right|\left(\frac{1}{\sqrt{n}}-\frac{1}{\sqrt{n+1}}\right),
\end{align*}
which yields that
\begin{align*}
\mathbb{P}_{\mu}\left(\sup_{t\in[n,n+1)}\left|\frac{I(t)}{\sqrt{t}}-\frac{I(n)}{\sqrt{n}}\right|\geq \varepsilon\right)&\leq \mathbb{P}_{\mu}\left(\left|I(n)\right|\left(\frac{1}{\sqrt{n}}-\frac{1}{\sqrt{n+1}}\right)\geq \frac{\varepsilon}{2}\right)\\
&=\mathbb{P}_{\mu}\left(\left|I(n)\right|\geq \frac{\varepsilon}{2\left(n^{-1/2}-(n+1)^{-1/2}\right)}\right).
\end{align*}
Finally, using the Chebyshev inequality and the fact that \hbox{$n^{-1/2}-(n+1)^{-1/2}\leq n^{-3/2}$}, we infer that, for very $n\geq n_0$,
$$
\mathbb{P}_{\mu}\left(\sup_{t\in[n,n+1)}\left|\frac{I(t)}{\sqrt{t}}-\frac{I(n)}{\sqrt{n}}\right|\geq \varepsilon\right)
\leq\frac{4\mathbb{E}_{\mu}\left(I^2(n)\right)}{n^3\varepsilon^2}\leq \frac{4(\|g\|_{\infty}n)^2}{n^3\varepsilon^2}=\frac{4\|g\|_{\infty}^2}{n\varepsilon^2}.
$$
Consequently, 
\begin{equation*}
\lim_{n\to\infty}\mathbb{P}_{\mu}\left(\sup_{t\in[n,n+1)}\left|\frac{I(t)}{\sqrt{t}}-\frac{I(n)}{\sqrt{n}}\right|\geq \varepsilon\right)=0,
\end{equation*}
which implies the desired claim.
\end{proof}

{
Additionally, as is evident from the following remark, we may assume without loss of generality that the initial distribution $\mu$ of $\Psi$ belongs to $\mathcal{M}_{1,2}^V(E)$.

\begin{remark}\label{rem:m1}
If \eqref{e:main} holds for all $\mu\in\{\delta_x:\,x\in E\}$, then it is valid for all $\mu\in\mathcal{M}_1(E)$. This follows directly from \eqref{e:pr_m} by applying the Lebesgue dominated convergence theorem.
\end{remark}
}

Another simple observation, to be used later, is expressing the Lyapunov condition assumed in \ref{(h2)} using the function $\mathcal{C}:E\to\mathbb{R}_+$, given by \eqref{eq:bar_C}.
\begin{lemma}\label{lem:P(t)C}
If $\{P(t)\}_{t\in\mathbb{R}_+}$ enjoys hypothesis \ref{(h2)} with some Borel measurable function $V:E\to\mathbb{R}_+$, then, for any $p\in(0,4]$, there exist constants $A_{p},B_{p}\geq 0$ and $\Gamma_{p}>0$ such that the map $\mathcal{C}:E\to\mathbb{R}_+$ given by \eqref{eq:bar_C} satisfies
$$
P(t)\mathcal{C}^p\left(x\right)\leq A_pe^{-\Gamma_p t}V^{p/2}(x)+B_p\;\;\;\text{for any}\;\;\;x\in E\;\;\;\text{and any}\;\;\;t\in\mathbb{R}_+.
$$
\end{lemma}

\begin{proof}
Let $x\in E$, $t\in\mathbb{R}_+$. Using the fact that, for any $r>0$, there exists a positive constant~$\zeta$ (precisely, $\zeta=1$ for $r\leq 1$ or $\zeta=2^{r-1}$ for $r>1$) such that 
\begin{equation}\label{simple_fact}
|a_1+a_2|^{r}\leq \zeta(|a_1|^{r}+|a_2|^{r})\;\;\;\text{for any}\;\;\;a_1,a_2\in\mathbb{R},
\end{equation}
we obtain $\mathcal{C}^p=\kappa^p(V+1)^{p/2}\leq \kappa^p \zeta (V^{p/2}+1)$ with some $\zeta>0$. Further,  the H\"older inequality (used with exponent $4/p$) yields that $P(t)V^{p/2}(x)\leq (P(t)V^2(x))^{p/4}$. Consequently, applying hypothesis \ref{(h2)} and again inequality \eqref{simple_fact} (with $r=p/4\in(0,1]$), we can conclude that 
\begin{align*}
P(t)\mathcal{C}^p\left(x\right)
&\leq \kappa^p\zeta\left(\left(P(t)V^2(x)\right)^{p/4}+1\right)\leq \kappa^p\zeta\left(\left(Ae^{-\Gamma t}V^{2}(x)+B\right)^{p/4}+1\right)\\
&\leq \kappa^p\zeta\left( A^{p/4}e^{-\frac{p\Gamma}{4} t}V^{p/2}(x)+B^{p/4}+1\right)=A_pe^{-\Gamma_p t}V^{p/2}(x)+B_p,
\end{align*}
where $A_p:=\kappa^p\zeta A^{p/4}$, $B_p:=\kappa^p\zeta(B^{p/4}+1)$ and $\Gamma_p:=(p\Gamma)/4$. The proof is now complete.
\end{proof}

As a straightforward consequence of Lemma \ref{lem:P(t)C}, we can deduce that, for any initial measure $\nu$ with finite second moment w.r.t. $V$, the fourth moment of the distribution of $\Psi$ w.r.t. $\mathcal{C}$,  is uniformly bounded over $t$.
\begin{corollary}\label{lem:sup_t}
Suppose that $\{P(t)\}_{t\in\mathbb{R}_+}$ fulfills hypotheses \ref{(h0)}-\ref{(h2)} with some continuous $V:E\to\mathbb{R}_+$. Then, for every $p\in(0,4]$ and the function $\mathcal{C}:E\to\mathbb{R}_+$ given by \eqref{eq:bar_C}, we have
$$
\sup_{t\in\mathbb{R}_+}
\left\langle
\mathcal{C}^p,\mu P(t)\right\rangle<\infty
\;\;\;\text{whenever}\;\;\; \mu\in\mathcal{M}_{1,2}^V(E), \quad\text{and}\quad  \<\mathcal{C}^p,\mu_*\><\infty.
$$
\end{corollary}

\begin{proof}
Let $p\in (0,4]$. Taking constants $A_p,B_p\geq 0$ and $\Gamma_p>0$ for which the assertion of Lemma \ref{lem:P(t)C} is valid, we get
\begin{equation}\label{eq:estimated}
\left\langle \mathcal{C}^p,\mu P(t)\right\rangle
=\left\langle P(t)\mathcal{C}^p,\mu \right\rangle
\leq A_p \left\langle V^{p/2},\mu\right\rangle+B_p<\infty
\end{equation}
for any $t\geq 0$ and $\mu\in\mathcal{M}_{1,2}^V(E)$. Since $\mu_*\in\mathcal{M}_{1,2}^V(E)$ due to Lemma \ref{lem:inv_exist}, applying the same estimation with $\mu=\mu_*$ also gives $\<\mathcal{C}^p,\mu_*\><\infty$.
\end{proof}

\subsection{A martingale-based decomposition for the given Markov process}\label{sec:df_mart}

Suppose that the transiton semigroup $\{P(t)\}_{t\in\mathbb{R}_+}$ of $\Psi$ fulfills hypotheses \ref{(h0)}-\ref{(h2)} with some continuous function $V:E\to\mathbb{R}_+$ (and therefore it is $V$-exponentially ergodic by Lemma \ref{lem:inv_exist}). Further, fix an arbitrary $g\in \operatorname{Lip}_{b,1}(E)$, and let $\bar{g}:=g-\langle g,\mu_*\rangle$. Then, by \eqref{e:erg_dfm}, we have
\begin{align}\label{eq:P(t)g_bound}
\begin{split}
|P(t)\bar{g}(x)|&=|\left\langle \bar{g},\delta_xP(t)\right\rangle|
=|\left\langle {g},\delta_xP(t)\right\rangle 
-\left\langle {g},\mu_*\right\rangle|\\
&\leq \|{g}\|_{\text{BL}}d_{\text{FM},\rho}\left(\delta_xP(t),\mu_*\right)
\leq \|{g}\|_{\text{BL}}\mathcal{C}(x)e^{-\gamma t} 
\end{split}
\end{align}
for every $x\in E$, $t\geq 0$ and some $\gamma>0$, 
whence 
\begin{equation}
\label{eq:new}
\int_0^{\infty}\left|P(t)\bar{g}(x)\right|dt
\leq \|{g}\|_{\text{BL}}\mathcal{C}(x)\int_0^{\infty}e^{-\gamma t}dt
=\frac{\|{g}\|_{\text{BL}}}{\gamma}\mathcal{C}(x)
\;\;\;\text{for any}\;\;\;x\in E.
\end{equation}
This, in turn, allows us to define the \emph{corrector function} $\chi:E\to\mathbb{R}$ as
\begin{equation}\label{def:chi}
\chi(x):=\int_0^{\infty}P(t)\bar{g}(x)dt\;\;\;\text{for any}\;\;\;x\in E.
\end{equation}

\begin{remark}  
\label{rem:cor_cont}
The function $\chi$ is continuous. Indeed, let $x_0\in E$. From \eqref{eq:P(t)g_bound} and the continuity of $\mathcal{C}$ it follows that there exists $\delta>0$ such that, for all $t\geq 0$ and any $x\in E$ with $\rho(x_0,x)<\delta$, we have
$$
|P(t)\bar{g}(x)|
\leq \|g\|_{\text{BL}}\mathcal{C}(x) e^{-\gamma t}\leq \|g\|_{\text{BL}}\left(\mathcal{C}\left(x_0\right)+1\right) e^{-\gamma t}.
$$
This, in turn, allows one to apply the Lebesgue dominated convergence theorem, which, together with \ref{(h0)}, gives \,$\lim_{x\to x_0} \chi(x)=\chi(x_0)$.
\end{remark}

\begin{remark} \label{rem:2}
Note that \eqref{eq:new} implies that, given $p>0$ {and $\mu\in\mathcal{M}_1(E)$}, we have
$$
\ew_{\mu}\left(| \chi(\Psi(t))|^p\right)
\leq \int_{E}\left(\int_0^{\infty}|P(s)\bar{g}(x)|\,ds\right)^p(\mu P(t))(dx)
\leq \left(\frac{\|{g}\|_{\text{BL}}}{\gamma}\right)^p\<\mathcal{C}^p,\, \mu P(t)\>
$$
for all $t\geq 0$.
\end{remark}

\begin{remark}\label{rem:fubini}
{It is also worth paying attention} to certain consequences of Fubini's theorem:\begin{itemize}
\item[(i)] Since the map $\mathbb{R}_+\times \Omega \ni (t,\omega)\mapsto \Psi(t)(\omega)$ is \hbox{$\mathcal{B}(\mathbb{R}_+)\otimes\mathcal{F}_{\infty}/\mathcal{B}(E)$}-measurable, it follows that, for any $\mu\in\mathcal{M}_1(E)$ and any $T>0$, we have 
$$\ew_{\mu}\left(\int_0^T \bar{g}(\Psi(t))\,dt \right)=\int_0^T \ew_{\mu}\left(\bar{g}(\Psi(t)) \right)\,dt.$$
\item[(ii)]The assumed product measurability of the map $\mathbb{R}_+\times E\ni(u,y)\mapsto P(u)\mathbbm{1}_A(y)$ with any \hbox{$A\in\mathcal{B}(E)$} (see Section \ref{sec:markov_operators}) easily implies the measurability of $(u,y)\mapsto P(u)\bar{g}(y)$, which, in turn, ensures that $\mathbb{R}_+\times \Omega \ni (u,\omega)\mapsto P(u)\bar{g}(\Psi(t)(\omega))$ is $\mathcal{B}(\mathbb{R}_+)\otimes \mathcal{F}(t)/\mathcal{B}(\mathbb{R})$-measurable for any $t\geq 0$. Moreover, it follows from \eqref{eq:P(t)g_bound} and Lemma \ref{lem:P(t)C} that the latter is also integrable with respect to $du\otimes\pr_x(d\omega)$ for every $x\in E$. Hence, using \eqref{eq:Ef}, for any $x\in E$ and $t\geq 0$, we have
\begin{align*}
P(t)\chi(x)&=\ew_x\left(\chi(\Psi(t))\right)=\ew_x\left(\int_0^{\infty}P(u)\bar{g}(\Psi(t))\,du \right)\\
&=\int_0^{\infty} \ew_{x}\left( P(u)\bar{g}(\Psi(t))\right)du=\int_0^{\infty} P(t+u)\bar{g}(x)du.
\end{align*}
\end{itemize} 
\end{remark}

Let us now define the processes
\begin{gather}
\label{def:M}
M(t):=\chi(\Psi(t))-\chi(\Psi(0))+\int_0^t\bar{g}(\Psi(s))ds\quad\text{for}\quad t\geq 0,\\
\label{def:R}
R(t):=\frac{1}{\sqrt{t}}\left(\chi(\Psi(0))-\chi(\Psi(t))\right)\quad\text{for}\quad t>0,
\end{gather}
and observe that
\begin{equation}\label{eq:martingale+rest}
\frac{1}{\sqrt{t}}\int_0^t\bar{g}\left(\Psi(s)\right)ds=\frac{M(t)}{\sqrt{t}}+R(t)\quad\text{for all}\quad t>0.
\end{equation}
Moreover, let $\{Z(n)\}_{n\in\mathbb{N}_0}$ stand for the sequence of the increments of $\{M(n)\}_{n\in\mathbb{N}_0}$, that is,
\begin{gather}\label{def:Z}
\begin{aligned}
Z(0)&:=M(0)=0,\\
Z(n)&:=M(n)-M(n-1)=\chi\left(\Psi(n)\right)
-\chi\left(\Psi(n-1)\right)
+\int_{n-1}^n\bar{g}\left(\Psi(s)\right)ds,\;\;n\in\mathbb{N}.
\end{aligned}
\end{gather}
It is worth noting here that $\sigma^2$, defined by \eqref{def:sigma}, can be now expressed as
\begin{equation}\label{eq:sigma}
\sigma^2=\ew_{\mu_*}\left(M^2(1)\right)= \ew_{\mu_*}\left(Z^2(1)\right).
\end{equation}

It is well-known that $\{M(t)\}_{t\in \mathbb{R}_+}$ is a martingale; see, e.g.,  \hbox{\cite[Proposition 5.2]{kom_walczuk}}. Nevertheless, for the convenience of a~reader, we provide a somewhat more detailed proof of this fact below, demonstrating the finiteness of $\sigma^2$ at the same time.

\begin{lemma}\label{lem:martingale}
Suppose that $\{P(t)\}_{t\in\mathbb{R}_+}$ satisfies hypotheses \ref{(h0)}-\ref{(h2)} with some continuous function $V:E\to\mathbb{R}_+$.  Then, for every $\mu\in\mathcal{M}_{1,2}^V(E)$, the process $\{M(t)\}_{t\in\mathbb{R}_+}$, given by \eqref{def:M}, is a~martingale in $\mathcal{L}^4(\pr_{\mu})$ w.r.t. to the natural filtration $\{\mathcal{F}(t)\}_{t\in\mathbb{R}_+}$ of $\Psi$. In particular, the variance $\sigma^2$, specified by \eqref{eq:sigma}, is then finite. 
\end{lemma}
\begin{proof}
Let $\mu\in\mathcal{M}_{1,2}^V(E)$ and $t\geq 0$. Using twice inequality \eqref{simple_fact} with $r=4$ and $\zeta=2^3$, we have
$$|M(t)|^4\leq \zeta^2\left(|\chi(\Psi(t))|^4+|\chi(\Psi(0))|^4\right)+\zeta(\left\|\bar{g}\right\|_{\infty}t)^4.$$
Hence, according to Remark \ref{rem:2}, it follows that
\begin{align}\label{eq:E_mu_M(t)}
\begin{split}
\mathbb{E}_{\mu}\left(|M(t)|^4\right)
&\leq \zeta^2\left(\ew_{\mu}\left(|\chi(\Psi(t))|^4\right)+\ew_{\mu}\left(|\chi(\Psi(0))|^4\right)\right)+\zeta(\left\|\bar{g}\right\|_{\infty}t)^4\\
&\leq \left(\frac{\zeta^{1/2}\|{g}\|_{\text{BL}}}{\gamma}\right)^4\left(
\left\langle\mathcal{C}^4,\mu P(t)\right\rangle
+\left\langle\mathcal{C}^4,\mu\right\rangle\right)+\zeta(\left\|\bar{g}\right\|_{\infty}t)^4,
\end{split}
\end{align}
which shows that $\mathbb{E}_{\mu}\left(|M(t)|^4\right)<\infty$ due to Corollary \ref{lem:sup_t} (applied for $p=4$). Clearly, since \hbox{$\mu_*\in \mathcal{M}_{1,2}^V(E)$} (according to Lemma \ref{lem:inv_exist}), this also yields the finiteness of $\sigma^2$. 

Now, let $s,t\in\mathbb{R}_+$ be such that $s<t$. Keeping in mind Remark \ref{rem:fubini}(i), we can write\vspace{-0.12cm}
\begin{align}\label{eq:mart_new}
\begin{split}
\mathbb{E}_{\mu}\left(M(t)|\mathcal{F}(s)\right)
=&
\mathbb{E}_{\mu}\left(\chi(\Psi(t))|\mathcal{F}(s)\right)
-\chi(\Psi(0))
+ \int_0^s \mathbb{E}_{\mu}\left(\bar{g}(\Psi(u))\,|\,\mathcal{F}(s)\right)du\\
&+ \int_s^t \mathbb{E}_{\mu}\left(\bar{g}(\Psi(u))\,|\,\mathcal{F}(s)\right)du\\
= &
\mathbb{E}_{\mu}\left(\chi(\Psi(t))|\mathcal{F}(s)\right)
-\chi(\Psi(0))
+ \int_0^s\bar{g}(\Psi(u))du\\
&+ \int_s^{\infty} \mathbb{E}_{\mu}\left( \bar{g}(\Psi(u))\,|\,\mathcal{F}(s)\right)du- \int_t^{\infty}\mathbb{E}_{\mu}\left(\bar{g}(\Psi(u))\,|\,\mathcal{F}(s)\right)du.
\end{split}
\end{align}
Then, applying the Markov property and \eqref{eq:Ef}, we see that\vspace{-0.1cm}
$$\mathbb{E}_{\mu}\left(\chi(\Psi(t))|\mathcal{F}(s)\right)
=\mathbb{E}_{\Psi(s)}\left(\chi(\Psi(t-s))\right)=P(t-s)\chi(\Psi(s)),$$
$$\int_s^{\infty} \mathbb{E}_{\mu}\left( \bar{g}(\Psi(u))\,|\,\mathcal{F}(s)\right)du=\int_0^{\infty}\mathbb{E}_{\Psi(s)}\left(\bar{g}(\Psi(u))\right)du=\int_0^{\infty}P(u)\bar{g}(\Psi(s))du=\chi(\Psi(s)),\vspace{-0.22cm}$$
\begin{align*}
\int_t^{\infty}\mathbb{E}_{\mu}\left(\bar{g}(\Psi(u))\,|\,\mathcal{F}(s)\right)du&=\int_{t-s}^{\infty}\mathbb{E}_{\Psi(s)}\left(\bar{g}(\Psi(u))\right)du=\int_{0}^{\infty}P(u+t-s)\bar{g}(\Psi(s))du\\
&=P(t-s)\chi(\Psi(s)),
\end{align*}
where the last equality follows from Remark \ref{rem:fubini}(ii). Consequently, returning to \eqref{eq:mart_new}, we finally infer that $\mathbb{E}_{\mu}\left(M(t)|\mathcal{F}(s)\right)=M(s)$, which completes the proof.
\end{proof}

Obviously, identity \eqref{eq:martingale+rest}, combined with Lemma \ref{lem:suffices}, reduces the proof of our main result to showing the CLT for the martingale $\{M(n)\}_{n\in\n_0}$, provided that $\{R(n)\}_{n\in\n}$ converges in law to $0$. This, in turn, follows from the following observation:
\begin{lemma}\label{lem:rest}
Suppose that $\{P(t)\}_{t\in\mathbb{R}_+}$  hypotheses \ref{(h0)}-\ref{(h2)} with some continuous function \hbox{$V:E\to\mathbb{R}_+$}. Then, for $\{R(t)\}_{t>0}$ given by \eqref{def:R} and any $\mu\in\mathcal{M}_{1,2}^V(E)$, we have \hbox{$R(t) \to 0$ in $\mathcal{L}^1(\mathbb{P}_{\mu})$} as $t\to \infty$. 
\end{lemma}
\begin{proof}
Taking into account Remark \ref{rem:2}, we see that
$$
\mathbb{E}_{\mu}|R(t)| \leq  \frac{\|{g}\|_{\text{BL}}}{\gamma\sqrt{t}}
\left(
\left\langle\mathcal{C},\mu\right\rangle
+\left\langle\mathcal{C},\mu P(t)\right\rangle
\right) \leq \frac{\|{g}\|_{\text{BL}}}{\gamma\sqrt{t}}
\left(
\left\langle\mathcal{C},\mu\right\rangle
+\sup_{s\in\mathbb{R}_+}\left\langle\mathcal{C},\mu P(s)\right\rangle
\right)  .
$$
which, in conjunction with Corollary \ref{lem:sup_t}, gives the desired claim.
\end{proof}

\subsection{Verification of the hypotheses of Theorem \ref{thm:martingale}}
The aim of this section is to show that the martingale $\{M(n)\}_{n\in\n_0}$, determined by \eqref{def:M}, fulfills the hypotheses \ref{cnd:m1}-\ref{cnd:m3} of Theorem \ref{thm:martingale}, and thus obeys the CLT. Before we proceed to verify these conditions, let us make the following crucial observation:
\begin{lemma}\label{lem:E_x_Z(i)^p}
Suppose that $\{P(t)\}_{t\in\mathbb{R}_+}$ satisfies hypotheses \ref{(h0)}-\ref{(h2)} with some continuous $V:E\to\mathbb{R}_+$, and let $\{Z(n)\}_{n\in\mathbb{N}}$ be given by~\eqref{def:Z}. Then, for every $p\in(0,4]$, there exists $\Gamma_p>0$ such that, for any $\mu\in\mathcal{M}_{1,2}^V(E)$ and certain constants $\tilde{A}_p,\tilde{B}_p\geq 0$ (depending on $\mu$), we have
\begin{equation}\label{eq:int mu P(t)}
\int_E\mathbb{E}_x\left(|Z(i)|^p\right)\,(\mu P(t))(dx)\leq \tilde{A}_p e^{-\Gamma_p (i+t)}+\tilde{B}_p\;\;\;\text{for all}\;\;\;i\in\mathbb{N}\;\;\;\text{and all}\;\;\;t\in\mathbb{R}_+,
\end{equation}
and, in particular,
\begin{equation}\label{eq:int mu*}
\sup_{i\in\n} \mathbb{E}_{\mu_*}\left(|Z(i)|^p\right)<\infty.
\end{equation}
\end{lemma}
\begin{proof}
Let $p\in (0,4]$, $i\in\mathbb{N}$ and $t\in\mathbb{R}_+$. Using inequality \eqref{simple_fact} we can choose $\zeta>0$ so that, for every $x\in E$,
\begin{align}\label{eq:estim}
\begin{split}
\mathbb{E}_{x}\left(|Z(i)|^{p}\right)
&\leq \zeta\mathbb{E}_{x}\left(\left|\chi(\Psi(i))-\chi(\Psi(i-1))\right|^p\right)
+\zeta\mathbb{E}_{x}\left(\left|\int_{i-1}^i\bar{g}(\Psi(s))ds\right|^{p}\right)\\
&\leq \zeta\mathbb{E}_{x}\left(\left|\chi(\Psi(i))-\chi(\Psi(i-1))\right|^{p}\right)
+\zeta\|\bar{g}\|_{\infty}^{p},
\end{split}
\end{align} 
and also
$$\mathbb{E}_{x}\left(\left|\chi(\Psi(i))-\chi(\Psi(i-1))\right|^{p}\right)\leq \zeta \left( \ew_x\left(\left|\chi(\Psi(i))\right|^p \right) +\ew_x\left(\left|\chi(\Psi(i-1))\right|^p \right)\right).$$
Consequently, appealing to Remark \ref{rem:2}, for any $x\in E$, we get
\begin{align}
\label{eq:estim_}
\begin{split}
\mathbb{E}_{x}\left(\left|\chi(\Psi(i))-\chi(\Psi(i-1))\right|^{p}\right)&\leq \zeta \left(\frac{\|g\|_{\text{BL}}}{\gamma}\right)^p
\left(\left\langle\mathcal{C}^p,\delta_x P(i)\right\rangle
+\left\langle\mathcal{C}^p,\delta_x P(i-1)\right\rangle
\right)\\
&=\zeta\left(\frac{\|g\|_{\text{BL}}}{\gamma}\right)^p
\left(
P(i)\mathcal{C}^p(x)+P(i-1)\mathcal{C}^p(x)
\right).
\end{split}
\end{align}
Now, take any constants \hbox{$\Gamma_p>0$} and $A_p,B_p\geq 0$ for which the assertion of Lemma \ref{lem:P(t)C} is valid, and let $\mu\in\mathcal{M}_{1,2}^V(E)$. Then
\begin{align}\label{eq:estim_cd}
\begin{split}
\int_E\mathbb{E}_x \big|\chi(\Psi(i))-&\chi(\Psi(i-1))\big|^{p}\,(\mu P(t))(dx)\\
&\leq \zeta\left(\frac{\|g\|_{\text{BL}}}{\gamma}\right)^p
\left\langle P(i)\mathcal{C}^p+P(i-1)\mathcal{C}^p,\mu P(t)\right\rangle\\
&=\zeta\left(\frac{\|g\|_{\text{BL}}}{\gamma}\right)^p
\left\langle P(i+t)\mathcal{C}^p+ P(i+t-1)\mathcal{C}^p,\mu\right\rangle\\
& \leq
\zeta\left(\frac{\|g\|_{\text{BL}}}{\gamma}\right)^p\left(
A_p\left(e^{-\Gamma_p(i+t)}+e^{-\Gamma_p(i+t-1)}\right)\left\langle V^{p/2},\mu\right\rangle+2B_p\right)
\\
& =
\zeta\left(\frac{\|g\|_{\text{BL}}}{\gamma}\right)^p\left(
A_p\left(1+e^{\Gamma_p}\right)e^{-\Gamma_p(i+t)}\left\langle V^{p/2},\mu\right\rangle+2B_p\right).
\end{split}
\end{align}
As a consequence of \eqref{eq:estim} and \eqref{eq:estim_cd}, we obtain
\begin{align*}
\int_E\mathbb{E}_x\left(|Z(i)|^p\right)\,(\mu P(t))(dx) &\leq\zeta^2\left(\frac{\|g\|_{\text{BL}}}{\gamma}\right)^p\left(
A_p\left(1+e^{\Gamma_p}\right)e^{-\Gamma_p(i+t)}\left\langle V^{p/2},\mu\right\rangle+2B_p\right)\\
&\quad+\zeta\|\bar{g}\|_{\infty}^{p},
\end{align*}
and, since $\<V^{p/2},\mu\><\infty$, we can finally deduce that \eqref{eq:int mu P(t)} holds with 
$$\tilde{A}_p:=\zeta^2\left(\frac{\|g\|_{\text{BL}}}{\gamma}\right)^pA_p(1+e^{\Gamma_p})\langle V^{p/2},\mu\rangle\quad\text{and}\quad {\tilde{B}_p:=2\zeta^2\left(\frac{\|g\|_{\text{BL}}}{\gamma}\right)^pB_p+\zeta\|\bar{g}\|_{\infty}^{p}}.$$
Obviously, \eqref{eq:int mu*} follows directly from \eqref{eq:int mu P(t)} (applied with $\mu=\mu_*$), since $\mu_*\in\mathcal{M}_{1,2}^V(E)$ by virtue of Lemma \ref{lem:inv_exist}.
\end{proof}

While proving that condition \ref{cnd:m2} holds, we will also need the continuity of the maps $E\ni x \mapsto \mathbb{E}_x\left(M^2(t)\right)$ (for every $t\in\mathbb{R}_+$), which is shown in the following two lemmas.

\begin{lemma}\label{lem:chi_continuous}
Suppose that $\{P(t)\}_{t\in\mathbb{R}_+}$ satisfies hypotheses \ref{(h0)}-\ref{(h2)} with some continuous $V:E\to\mathbb{R}_+$. Then the map $E\ni x\mapsto\mathbb{E}_x(\chi(\Psi(t)))$ is also continuous for any $t\in\mathbb{R}_+$.
\end{lemma}
\begin{proof}

Let $t\in\mathbb{R}_+$, $x_0\in E$, and observe that \eqref{eq:P(t)g_bound} gives
$$
|P(t+s)\bar{g}(x)|=|P(t)\left(P(s)\bar{g}\right)(x)|\leq \|g\|_{\text{BL}}P(t)\mathcal{C}(x) e^{-\gamma s}\;\;\;\text{for any}\;\;\;x\in E,\;s\in\mathbb{R}_+.
$$ 
Hence, referring to Lemma \ref{lem:P(t)C} and using the continuity of $V^{1/2}$, we can choose constants $A_1,B_1\geq 0$ and $\Gamma_1>0$, as well as $\delta>0$, so that, for any $u,v\in\mathbb{R}_+$ and all $x\in E$ satisfying~$\rho(x_0,x)<\delta$,
$$
|P(t+s)\bar{g}(x)|\leq \|g\|_{\text{BL}}\left(A_1e^{-\Gamma_1 t}\left(V^{1/2}\left(x_0\right)+1\right)+B_1\right) e^{-\gamma s}.
$$ 
Consequently, having in mind Remark \ref{rem:fubini}(ii) and \eqref{eq:Ef}, we can apply the Lebesgue dominated convergence theorem to conclude that, for any $x_0\in E$,
$$
\lim_{x\to x_0}\mathbb{E}_x\left(\chi\left(\Psi(t)\right)\right)
=\int_0^{\infty}\lim_{x\to x_0}P(t+s)\bar{g}(x)\,ds
=\int_0^{\infty}P(t+s)\bar{g}(x_0)\,ds=\mathbb{E}_{x_0}\left(\chi\left(\Psi(t)\right)\right),
$$
where the second equality follows from \ref{(h0)}. The proof is now  completed.
\end{proof}

\begin{lemma}\label{lem:H_k_continuous}
Suppose that $\{P(t)\}_{t\in\mathbb{R}_+}$ satisfies hypotheses \ref{(h0)}-\ref{(h2)} with some continuous $V:E\to\mathbb{R}_+$. Then, for every $t\in\mathbb{R}_+$, the map \hbox{$E\ni x\mapsto \mathbb{E}_x(M^2(t))$}, with $M(t)$ given by~\eqref{def:M}, is continuous.
\end{lemma}

\begin{proof}
Let $t\in\mathbb{R}_+$ and observe that, for any $x\in E$, we have
\begin{align*}
\mathbb{E}_x&\left(M^2(t)\right)
=\mathbb{E}_x
\left(
\left(
\chi\left(\Psi(t)\right)-\chi\left(\Psi(0)\right)
+\int_0^t\bar{g}\left(\Psi(s)\right)\,ds
\right)^2
\right)\\
&=\mathbb{E}_x
\left(\chi^2\left(\Psi(t)\right)\right)
+\chi^2\left(x\right)+\mathbb{E}_x
\left(\left(\int_0^t\bar{g}\left(\Psi(s)\right)\,ds
\right)^2\right)\\
&\quad
-2\chi\left(x\right)\mathbb{E}_x
\left(\chi\left(\Psi(t)\right)\right)+2\mathbb{E}_x
\left(\chi\left(\Psi(t)\right)\int_0^t\bar{g}\left(\Psi(s)\right)\,ds
\right)
-2\chi\left(x\right)\mathbb{E}_x
\left(\int_0^t\bar{g}\left(\Psi(s)\right)\,ds
\right).
\end{align*}
Consequently, in view of Lemma \ref{lem:chi_continuous} and the continuity of $\chi$ (demonstrated in Remark \ref{rem:cor_cont}), it suffices to show the continuity of the following maps:
$$
x\mapsto\mathbb{E}_x\left(\chi^2\left(\Psi(t)\right)\right),\;\; x\mapsto\mathbb{E}_x
\left(\left(\int_0^t\bar{g}\left(\Psi(s)\right)\,ds
\right)^2\right),\;\;
x\mapsto\mathbb{E}_x\left(\chi\left(\Psi(t)\right)\int_0^{t}\bar{g}\left(\Psi(s)\right)\,ds\right).
$$

According to \eqref{eq:P(t)g_bound} we have 
$$|\left(P(u)\bar{g}\cdot P(v)\bar{g}\right)(x)|\leq \|g\|_{\text{BL}}^2e^{-\gamma u-\gamma v}\mathcal{C}^2(x)\quad\text{for any}\quad u,v\in\mathbb{R}_+,\;x\in E.$$
Hence, using the Fubini theorem in a manner similar to that in Remark \ref{rem:fubini}(ii) gives
\begin{align*}
\mathbb{E}_{x}\left(\chi^2\left(\Psi(t)\right)\right)
&=\mathbb{E}_{x}\left(\int_0^{\infty}\int_0^{\infty} P(u)\bar{g}(\Psi(t))\cdot P(v)\bar{g}(\Psi(t))\,du\,dv\right)\\
&=\int_0^{\infty}\int_0^{\infty}\mathbb{E}_{x}\left((P(u)\bar{g}\cdot P(v)\bar{g})(\Psi(t))\right)\,du\,dv\\
&=\int_0^{\infty}\int_0^{\infty}P(t)\left(P(u)\bar{g}\cdot P(v)\bar{g}\right)\left(x\right)\,du\,dv\quad\text{for all}\quad x\in E.
\end{align*}
Let us now fix an arbitrary $x_0\in E$. Taking any constants $A_2,B_2\geq 0$ (and $\Gamma_2>0$) for which the assertion of Lemma \ref{lem:P(t)C} is valid with $p=2$, and having in mind the continuity of~$V$, we can choose $\delta>0$ such that, for any $u,v\in\mathbb{R}_+$ and all $x\in E$ with $\rho(x_0,x)<\delta$,
$$
|P(t)\left(P(u)\bar{g}\cdot P(v)\bar{g}\right)\left(x\right)|
\leq \|g\|_{\text{BL}}^2e^{-\gamma (u+v)}\left(A_2(V(x_0)+1)+B_2\right).
$$
Applying the Lebesgue dominated convergence theorem we can therefore conclude that
\begin{align*}
\lim_{x\to x_0}\mathbb{E}_x\left(\chi^2\left(\Psi(t)\right)\right)
&=\int_0^{\infty}\int_0^{\infty}\lim_{x \to x_0}P(t)\left(P(u)\bar{g}\cdot P(v)\bar{g}\right)(x)\,du\,dv\\
&=\int_0^{\infty}\int_0^{\infty}P(t)\left(P(u)\bar{g} \cdot P(v)\bar{g}\right)\left(x_0\right)\,du\,dv
=\mathbb{E}_{x_0}\left(\chi^2\left(\Psi(t)\right)\right),
\end{align*}
where the penultimate equality follows from \ref{(h0)}. Hence \hbox{$x\mapsto\mathbb{E}_x\left(\chi^2\left(\Psi(t)\right)\right)$} is continuous at $x_0$, and thus in any $x\in E$.

To prove the continuity of the remaining two maps, let us first note that a direct application of Fubini's theorem yields that, for any $t\in\mathbb{R}_+$ and every integrable $f:[0,t]\to\mathbb{R}$,
$$
\left(\int_0^tf(s)\,ds\right)^2=2\int_0^t\int_0^vf(u)f(v)\,du\,dv.
$$

As a consequence, arguing analogously as in Remark \ref{rem:fubini}(i), we get
\begin{align*}
\mathbb{E}_x\left(\left(\int_0^t\bar{g}\left(\Psi(s)\right)\,ds
\right)^2\right)
&=2\mathbb{E}_x\left(\int_0^t\int_0^v\bar{g}\left(\Psi(u)\right)\bar{g}\left(\Psi(v)\right)\,du\,dv\right)\\
&=2\int_0^t\int_0^v\mathbb{E}_x\left(\bar{g}\left(\Psi(u)\right)\bar{g}\left(\Psi(v)\right)\right)\,du\,dv \quad\text{for any}\quad x\in E.
\end{align*}
Further, referring to the properties of conditional expectation, the Markov property and identity \eqref{eq:Ef}, we obtain
\begin{align*}
\mathbb{E}_x\left(\left(\int_0^t\bar{g}\left(\Psi(s)\right)\,ds
\right)^2\right)
&=2\int_0^t\int_0^v\mathbb{E}_x
\Big(\bar{g}\left(\Psi(u)\right)\mathbb{E}_x\left(\bar{g}\left(\Psi(v)\right)\,|\,\mathcal{F}(u)\right)\Big)\,du\,dv\\
&=2\int_0^t\int_0^v\mathbb{E}_x
\Big(\bar{g}\left(\Psi(u)\right)
\mathbb{E}_{\Psi(u)}\left(\bar{g}(\Psi(v-u))\right)\Big)\,du\,dv\\
&=2\int_0^t\int_0^v\mathbb{E}_x
\Big(\left(\bar{g}\cdot P(v-u)\bar{g}\right)(\Psi(u))\Big)\,du\,dv\\
&=2\int_0^t\int_0^vP(u)\left(\bar{g}\cdot P(v-u)\bar{g}\right)(x)\,du\,dv\;\;\;\text{for every}\;\;\;x\in E.
\end{align*}
Let $x_0\in E$. Proceeding analogously as before, we can now use \eqref{eq:P(t)g_bound} and then refer to Lemma~\ref{lem:P(t)C} and the continuity of $V^{1/2}$ to finally conclude that there exist constants \hbox{$A_1,B_1\geq 0$} and some $\delta>0$ such that, for any $u,v\in\mathbb{R}_+$ and all $x\in E$ satisfying $\rho(x_0,x)<\delta$,
\begin{equation}\label{e:double}
|P(u)\left(\bar{g}P(v-u)\bar{g}\right)\left(x\right)|
\leq \|g\|_{\text{BL}}^2e^{-\gamma (v-u)}\left(A_1(V^{1/2}(x_0)+1)+B_1\right).
\end{equation}
This enables us to apply the Lebesgue dominated convergence theorem, which yields that
\begin{align*}
\lim_{x\to x_0}\mathbb{E}_x\left(\left(\int_0^t\bar{g}\left(\Psi(s)\right)\,ds
\right)^2\right)
&=2\int_0^t\int_0^v\lim_{x\to x_0}P(u)\left(\bar{g}P(v-u)\bar{g}\right)\left(x\right) du\,dv\\
&=2\int_0^t\int_0^vP(u)\left(\bar{g}P(v-u)\bar{g}\right)\left(x_0\right)du\,dv\\
&=\mathbb{E}_{x_0}\left(\left(\int_0^t\bar{g}\left(\Psi(s)\right)\,ds
\right)^2\right),
\end{align*}
where the penultimate equality follows from \ref{(h0)}. In view of arbitrariness of $x_0$, we have thus shown that the map \hbox{$x\mapsto\mathbb{E}_x((\int_0^t\bar{g}(\Psi(s))\,ds)^2)$} is continuous.

Finally, taking into account that
$$|\bar{g}(\Psi(u))\cdot P(v)\bar{g}(x)|\leq \norma{g}_{\mathrm{BL}}^2 \mathcal{C}(x)e^{-\gamma t}\quad\text{for any}\quad u,v\in\mathbb{R}_+,\;x\in E,$$
due to \eqref{eq:P(t)g_bound}, and reasoning similarly as in Remark \ref{rem:fubini}, we see that, for any $x\in E$,
\begin{align*}
\mathbb{E}_x\left(\chi\left(\Psi(t)\right)\int_0^{t}\bar{g}\left(\Psi(s)\right)\,ds\right)
&=\mathbb{E}_x\left(\int_0^{\infty}\int_0^t \bar{g}(\Psi(u))\cdot P(v)\bar{g}(\Psi(t))\,du\,dv\right)
\\
&=\int_0^{\infty}\int_0^t\mathbb{E}_x\left(\bar{g}(\Psi(u))\cdot P(v)\bar{g}(\Psi(t))\right)du\,dv\\
&=\int_0^{\infty}\int_0^t\mathbb{E}_x\left(\bar{g}(\Psi(u))\,\mathbb{E}_x\left(P(v)\bar{g}(\Psi(t))\,|\,\mathcal{F}(u)\right)\right) du\,dv\\
&=\int_0^{\infty}\int_0^t\mathbb{E}_x\left(\bar{g}(\Psi(u))\,\mathbb{E}_{\Psi(u)}\left(P(v)\bar{g}(\Psi(t-u))\right)\right) du\,dv\\
&=\int_0^{\infty}\int_0^tP(u)\left(\bar{g}P(t-u+v)\bar{g}\right)(x)\,du\,dv.
\end{align*}
Hence, again, fixing $x_0\in E$ and referring sequentially  to \eqref{eq:P(t)g_bound}, Lemma \ref{lem:P(t)C} and the continuity of $V^{1/2}$, we can estimate the integrand on the right-hand similarly as in \eqref{e:double} (with $t+v$ in place of $v$). This, analogously as before, enables the use of the Lebesgue theorem to deduce the continuity of $x\mapsto\mathbb{E}_x(\chi(\Psi(t))\int_0^{t}\bar{g}(\Psi(s))\,ds)$. The proof is now complete.
\end{proof}

We can now proceed to verifying hypothesis \ref{cnd:m1}-\ref{cnd:m3}.\begin{lemma}\label{lem:M1}
Suppose that $\{P(t)\}_{t\in\mathbb{R}_+}$ satisfies hypotheses \ref{(h0)}-\ref{(h2)} with some continuous $V:E\to\mathbb{R}_+$, and that $\mu\in\mathcal{M}_{1,2}^V(E)$. Then, the sequence $\{Z(n)\}_{n\in\mathbb{N}}$ of martingale increments, given by \eqref{def:Z}, fulfills property \ref{cnd:m1} with $\ew=\ew_{\mu}$.
\end{lemma}

\begin{proof}
Let $\varepsilon>0$ and $N\in\n$. By the Markov property, for every $n\in\mathbb{N}_0$,
\begin{equation*}
\mathbb{E}_{\mu}\left(Z^2(n+1)\mathbbm{1}_{\left\{Z(n+1)\geq\varepsilon\sqrt{N}\right\}}|\mathcal{F}(n)\right)
=\mathbb{E}_{\Psi(n)}\left(Z^2(1)\mathbbm{1}_{\left\{Z(1)\geq\varepsilon\sqrt{N}\right\}}\right),
\end{equation*}
whence, using the properties of the conditional expectation, we get
\begin{align*}
\mathbb{E}_{\mu}\left(
Z^2(n+1)\mathbbm{1}_{\left\{Z(n+1)\geq\varepsilon\sqrt{N}\right\}}
\right)
&=\mathbb{E}_{\mu}
\left(
\mathbb{E}_{\mu}
\left(
Z^2(n+1)\mathbbm{1}_{\left\{Z(n+1)\geq\varepsilon\sqrt{N}\right\}}|\mathcal{F}(n)
\right)
\right)\\
&=\mathbb{E}_{\mu}\left(
\mathbb{E}_{\Psi(n)}\left(Z^2(1)\mathbbm{1}_{\left\{Z(1)\geq\varepsilon\sqrt{N}\right\}}\right)
\right)\\ 
&=\int_E\mathbb{E}_x\left(Z^2(1)\mathbbm{1}_{\left\{Z(1)\geq\varepsilon\sqrt{N}\right\}}\right)(\mu P(n))(dx).
\end{align*}
This allows us to write
\begin{equation}\label{eq:short}
\frac{1}{N}\sum_{n=0}^{N-1}\mathbb{E}_{\mu}\left(Z^2(n+1)\mathbbm{1}_{\left\{\left|Z(n+1)\right|\geq \varepsilon \sqrt{N}\right\}}\right)=\left\langle G_N,\mu U_N\right\rangle,
\end{equation}
where 
\begin{gather*}
G_N(x):=\mathbb{E}_x\left(Z^2(1)\mathbbm{1}_{\left\{|Z(1)|\geq\varepsilon\sqrt{N}\right\}}\right)
\;\;\;\text{and}\;\;\;
\mu U_N:=\frac{1}{N}\sum_{n=0}^{N-1}\mu P(n)\;\;\;\text{for}\;\;\;x\in E.
\end{gather*}
Now, let $\delta\in(0,2]$. Using subsequently the H\"{o}lder inequality (with exponent \hbox{$(2+\delta)/2$}) and the Markov inequality, we obtain
\begin{align*}
G_N(x)
&=\mathbb{E}_x\left(Z^2(1)\mathbbm{1}_{\left\{|Z(1)|\geq\varepsilon\sqrt{N}\right\}}\right)
\leq \left(\mathbb{E}_x\left(|Z(1)|^{2+\delta}\right)\right)^{2/(2+\delta)}
\mathbb{P}_{x}\left(|Z(1)|\geq \varepsilon\sqrt{N}\right)^{\delta/(2+\delta)}\\
&\leq \left(\mathbb{E}_x\left(|Z(1)|^{2+\delta}\right)\right)^{2/(2+\delta)}
\left(\left(\varepsilon\sqrt{N}\right)^{-(2+\delta)}\mathbb{E}_x\left(|Z(1)|^{2+\delta}\right)\right)^{\delta/ (2+\delta)}\\
&=\left(\varepsilon\sqrt{N}\right)^{-\delta}\mathbb{E}_x\left(|Z(1)|^{2+\delta}\right)
\end{align*} 
for all $x\in E$, which, in turn, implies that
\begin{align*}
\left\langle G_N,\mu U_N\right\rangle
\leq \left(\varepsilon\sqrt{N}\right)^{-\delta}
\int_E \mathbb{E}_x\left(|Z(1)|^{2+\delta}\right)\mu U_N(dx).
\end{align*}
On the other hand, Lemma \ref{lem:E_x_Z(i)^p} guarantees that\; $\sup_{N\in\mathbb{N}}\int_E\mathbb{E}_x\left(|Z(1)|^{2+\delta}\right)\,\mu U_N(dx)<\infty$. Hence $\lim_{N\to\infty} \<G_N,\, \mu U_N\>=0$, which, due to \eqref{eq:short}, gives the desired claim.
\end{proof}

\begin{lemma}\label{lem:M2}
Suppose that $\{P(t)\}_{t\in\mathbb{R}_+}$ satisfies hypotheses \ref{(h0)}-\ref{(h2)} with some continuous $V:E\to\mathbb{R}_+$, and that $\mu\in\mathcal{M}_{1,2}^V(E)$.  Then, the sequence $\{Z(n)\}_{n\in\mathbb{N}}$ of martingale increments, given by \eqref{def:Z}, fulfills property \ref{cnd:m2} with $\sigma^2$ specified by \eqref{eq:sigma}, and $\mathbb{E}=\mathbb{E}_{\mu}$.
\end{lemma}

\begin{proof}
Obviously, the first part of condition \ref{cnd:m2}, { including the finiteness of $\sigma^2$}, follows directly from Lemma \ref{lem:E_x_Z(i)^p}.

Keeping in mind \eqref{def:<m>} and using the properties of the conditional expectation, as well as the Markov property, for any $j,k\in\n$, we obtain
\begin{align*}
\mathbb{E}_{\mu}&\left|\frac{1}{k}\mathbb{E}_{\mu}\left(\left\langle M\right\rangle_{jk}-\langle M\rangle_{(j-1)k}\big|\mathcal{F}((j-1)k)\right)-\sigma^2\right|\\
&=\mathbb{E}_{\mu}
\left|
\frac{1}{k}
\sum_{i=(j-1)k+1}^{jk}
\mathbb{E}_{\mu}\left(Z^2(i)|\mathcal{F}((j-1)k)\right)-\sigma^2\right|\\
&
=\mathbb{E}_{\mu}
\left|
\frac{1}{k}
\sum_{i=(j-1)k+1}^{jk}
\mathbb{E}_{\Psi((j-1)k)}\left(Z^2(i-(j-1)k)\right)-\sigma^2\right|\\
&
=\mathbb{E}_{\mu}
\left|
\frac{1}{k}
\sum_{i=1}^{k}
\mathbb{E}_{\Psi((j-1)k)}\left(Z^2(i)\right)-\sigma^2\right|
=\int_{E}\left|
\mathbb{E}_{x}
\left(
\frac{1}{k}
\sum_{i=1}^{k}
Z^2(i)\right)-\sigma^2\right|(\mu P((j-1)k))(dx).
\end{align*}
Hence, for any given $l,k\in\mathbb{N}$, we may write
\begin{equation}
\label{eq:M2_1}
\frac{1}{l}\sum_{j=1}^l\mathbb{E}_{\mu}\left|\frac{1}{k}\mathbb{E}_{\mu}\left(\left\langle M\right\rangle_{jk}-\langle M\rangle_{(j-1)k}\big|\mathcal{F}((j-1)k)\right)-\sigma^2\right|
=\left\langle \left|H_k\right|,\frac{1}{l}\sum_{j=1}^l\mu P((j-1)k)\right\rangle,
\end{equation}
where
\begin{equation}\label{def:H_k}
H_k(x)=\mathbb{E}_x\left(\frac{1}{k}\sum_{i=1}^kZ^2(i)\right)-\sigma^2=\frac{1}{k}\mathbb{E}_x\left(M^2(k)\right)-\sigma^2\;\;\;\text{for every}\;\;\;x\in E.
\end{equation}
Therefore, it now suffices to show that
\begin{equation}\label{e0}
\lim_{k\to\infty}\limsup_{l\to\infty}\left\langle \left|H_k\right|,\frac{1}{l}\sum_{j=1}^l\mu P((j-1)k)\right\rangle=0.
\end{equation}
To do this, we will use Lemma \ref{lem:bogachev} and the Birkhoff ergodic theorem.

Let us first observe that, according to Lemma \ref{lem:H_k_continuous}, $H_k$ is continuous for every $k\in\mathbb{N}$. Further, fix $k\in\mathbb{N}$, $q\in(1,2]$, and let $\Gamma_{2q}>0$ and $\tilde{A}_{2q},\tilde{B}_{2q}\geq 0$ be the constants for which assertion \eqref{eq:int mu P(t)} of Lemma \ref{lem:E_x_Z(i)^p} is valid with the given $\mu$. Then, using \eqref{simple_fact} with $r=q$ and $\zeta=2^{q-1}$, as well as the Jensen inequality, for every $l\in\n$, we obtain
\begin{align*}
\Big<|H_k|^q ,\, \frac{1}{l}\sum_{j=1}^l\mu P((j-1)k)\Big>
&\leq \frac{\zeta}{l} \sum_{j=1}^l 
\int_E  \left( \left(\ew_x\left(\frac{1}{k}\sum_{i=1}^k Z^2(i) \right)\right)^q
+\sigma^{2q}\right) 
\mu P((j-1)k)(dx)\\
&
\leq \zeta\left( \frac{1}{l} \sum_{j=1}^l 
\int_E  \ew_x\left(\frac{1}{k}\sum_{i=1}^k Z^{2q}(i) \right)\mu P((j-1)k)(dx)+\sigma^{2q}\right)\\
&
= \zeta\left( \frac{1}{lk} \sum_{j=1}^l \sum_{i=1}^k \int_E  \ew_x\left(Z^{2q}(i) \right)\mu P((j-1)k)(dx)+\sigma^{2q}\right)\\
&\leq \zeta(\tilde{A}_{2q}+\tilde{B}_{2q}+\sigma^{2q})<\infty.
\end{align*}

Moreover, referring to the $V$-ergodicity of $\{P(t)\}_{t\in\mathbb{R}_+}$ (established in Lemma \ref{lem:inv_exist}), as well as the Ces\'aro mean convergence theorem, we see that \hbox{$\{l^{-1}\sum_{j=1}^l\mu P((j-1)k)\}_{l\in\n}$} converges weakly to $\mu_*$. This observation, together with the above estimate and the continuity of $H_k$, enables us to use Lemma \ref{lem:bogachev}, which gives
\begin{equation}
\label{e1}
\lim_{l\to \infty} \<|H_k|,\,\frac{1}{l}\sum_{j=1}^l \mu P((j-1)k)\>=\<|H_k|, \mu_*\>.
\end{equation}

In view of the above, it remains to prove that
$\lim_{k\to\infty}\<|H_k|, \mu_*\>=0$. 
For this aim, consider the subfamily $\{\Theta_k\}_{k\in\n_0}$ of the shift operators, defined according to \eqref{e:shift}, and note that $\Theta_k=\Theta_1^k$ for every $n\in\n_0$. Since $\mu_*$ is the unique invariant distribution of $\Psi$, it follows that $\Theta_1$ is an (\hbox{$\mathcal{F}_{\infty}$-measurable}) ergodic transformation preserving the measure $\pr_{\mu_*}$, i.e., $\pr_{\mu_*}(\Theta_1^{-1}(F))=\pr_{\mu_*}(F)$ for every $F\in\mathcal{F}_{\infty}$, and $\pr_{\mu_*}(F)\in\{0,1\}$ whenever $F\in\mathcal{F}_{\infty}$ satisfies $\Theta_1^{-1}(F)=F$. Moreover, it is easily seen that $Z^2(k)=Z^2(1)\circ \Theta_{k-1}=Z^2(1)\circ \Theta_1^{k-1}$ for any $k\in\n$, and by Lemma \ref{lem:E_x_Z(i)^p} we know that $Z^2(1)\in\mathcal{L}^1(\pr_{\mu_*})$. Consequently, the von Neumann mean ergodic theorem adapted to $\mathcal{L}^p$ spaces (see \cite[Corollary 1.14.1]{Walters} or \cite[Theorem 1.4]{krengel} and the discussion after~it), together with \hbox{\cite[Proposition 1.6]{krengel}}, ensures that
$$\lim_{k\to\infty}\frac{1}{k}\sum_{i=1}^k Z^2(i)=\lim_{k\to\infty}\frac{1}{k}\sum_{i=0}^{k-1} Z^2(1)\circ \Theta_1^i=\ew_{\mu_*}[Z^2(1)]=\sigma^2\quad\text{in}\quad\mathcal{L}^1(\pr_{\mu_*}).$$
Finally, taking into account that
$$0\leq \<|H_k|, \mu_*\>\leq \int_E  \ew_x\left|\frac{1}{k}\sum_{i=1}^k Z^2(i)-\sigma^2 \right|  \mu_*(dx)=\ew_{\mu_*}\left|\frac{1}{k}\sum_{i=1}^k Z^2(i)-\sigma^2\right|,$$
for every $k\in\n$, we get
$\lim_{k\to\infty} \<|H_k|, \mu_*\>=0$. Obviously, in view of \eqref{e1}, we can now conclude that \eqref{e0} holds, which completes the proof.
\end{proof}

\begin{lemma}\label{lem:M3}
Suppose that $\{P(t)\}_{t\in\mathbb{R}_+}$ satisfies hypotheses \ref{(h0)}-\ref{(h2)} with some continuous $V:E\to\mathbb{R}_+$, and that $\mu\in\mathcal{M}_{1,2}^V(E)$.  
Then, the sequence $\{Z(n)\}_{n\in\mathbb{N}}$ of martingale increments, given by \eqref{def:Z}, enjoys property \ref{cnd:m3} with $\ew=\ew_{\mu}$.
\end{lemma}

\begin{proof}
Let $\varepsilon>0$ and fix $k,l,j\in\mathbb{N}$, $n\geq (j-1)k$ arbitrarily. The Markov property yields~that
\begin{align*}
\mathbb{E}_{\mu}\Big(\big(1+Z^2(n+1)&\big)\mathbbm{1}_{\left\{|M(n)-M((j-1)k)|\geq \varepsilon \sqrt{kl}\right\}}|\mathcal{F}((j-1)k)\Big)\\
&=\mathbb{E}_{\mu}\left(\left(1+Z^2(n+1)\right)\mathbbm{1}_{\left\{|Z((j-1)k+1)+\ldots+ Z(n)|\geq \varepsilon \sqrt{kl}\right\}}|\mathcal{F}((j-1)k)\right)\\
&=\mathbb{E}_{\Psi((j-1)k)}\left(\left(1+Z^2(n+1-(j-1)k)\right)\mathbbm{1}_{\left\{|Z(1)+\ldots+ Z(n-(j-1)k)|\geq \varepsilon \sqrt{kl}\right\}}\right)\\
&=\mathbb{E}_{\Psi((j-1)k)}\left(\left(1+Z^2(n+1-(j-1)k)\right)\mathbbm{1}_{\left\{|M(n-(j-1)k)|\geq \varepsilon \sqrt{kl}\right\}}\right).
\end{align*}
Thus, taking the expectation of both sides we get 
\begin{align*}
&\mathbb{E}_{\mu}\left(\left(1+Z^2(n+1)\right)\mathbbm{1}_{\left\{|M(n)-M((j-1)k)|\geq \varepsilon \sqrt{kl}\right\}}\right)\\
&\qquad\qquad=\int_E
\mathbb{E}_{x}
\left(
\left(1+Z^2(n+1-(j-1)k)\right)
\mathbbm{1}_{\left\{|M(n-(j-1)k)|\geq \varepsilon \sqrt{kl}\right\}}
\right)
\left(\mu P((j-1)k)\right)(dx).
\end{align*}
Now, summing up the above equality over all $n=(j-1)k,\ldots,jk-1$, and then over all $j=1,\ldots l$, and finally dividing both sides of the resulting identity by $kl$, we obtain
\begin{align*}
\frac{1}{kl}\sum_{j=1}^l\sum_{n=(j-1)k}^{jk-1}\mathbb{E}_{\mu}\Big(\left(1+Z^2(n+1)\right)&\mathbbm{1}_{\left\{|M(n)-M((j-1)k)|\geq \varepsilon \sqrt{kl}\right\}}\Big)
\\
&=\frac{1}{kl}\sum_{j=1}^l\sum_{n=0}^{k-1}\left\langle G_{k,l,n},\, \mu P((j-1)k)\right\rangle,
\end{align*}
where
$$
G_{k,l,n}(x)=\mathbb{E}_x\left(\left(1+Z^2(n+1)\right)\mathbbm{1}_{\left\{|M(n)|\geq \varepsilon\sqrt{kl}\right\}}\right)\;\;\;\text{for}\;\;\;x\in E.
$$
It is now clear that, to end the proof, we need to show that
$$ \lim_{k\to\infty} \limsup_{l\to \infty} \frac{1}{kl} \sum_{j=1}^l \sum_{n=0}^{k-1} \langle G_{k,l,n},\, \mu P((j-1)k) \rangle = 0.$$

From the H\"older inequality it follows that, for any $x\in E$, $k,l\in\n$ and $n\in\n_0$, we have
\begin{align*}
G_{k,l,n}(x)
&\leq 
\left(\mathbb{E}_x\left(\left(1+Z^2(n+1)\right)^2\right)\right)^{1/2}
\mathbb{P}_x\left(|M(n)|\geq \varepsilon\sqrt{kl}\right)^{1/2}\\
&=\left(1+2\mathbb{E}_x\left(Z^2(n+1)\right)+\mathbb{E}_x\left(Z^4(n+1)\right)\right)^{1/2}
\mathbb{P}_x\left(|M(n)|\geq \varepsilon\sqrt{kl}\right)^{1/2},
\end{align*}
which (again by the H\"older inequality) yields that, for any $k,l\in\n$,
\begin{align}\label{eq:last_lemma_2}
&\frac{1}{kl}\sum_{j=1}^l\sum_{n=0}^{k-1}\left\langle G_{k,l,n},\,\mu P((j-1)k)\right\rangle \nonumber\\
&\qquad\leq  \frac{1}{kl}\sum_{j=1}^l\sum_{n=0}^{k-1}
\left(\int_E\left(1+2\mathbb{E}_x\left(Z^2(n+1)\right)+\mathbb{E}_x\left(Z^4(n+1)\right)\right)\,(\mu P((j-1)k))(dx)\right)^{1/2}\nonumber\\
&\qquad\quad\times\left(\int_E \mathbb{P}_x\left(|M(n)|\geq \varepsilon\sqrt{kl}\right)\,(\mu P((j-1)k))(dx)\right)^{1/2}.
\end{align}
Now, fix $k,l,j\in\n$ and $n\in\n_0$. Taking, for $p\in\{2,4\}$, any constants $\Gamma_p>0$ and $\tilde{A}_p, \tilde{B}_p\geq 0$ for which assertion \eqref{eq:int mu P(t)} of Lemma \ref{lem:E_x_Z(i)^p} is valid with $p$, we see that the first integral on the right hand-side of \eqref{eq:last_lemma_2} can be estimated by
\begin{align}\label{eq:last_lemma_3}
\begin{aligned}
&\int_E\left(1+2\mathbb{E}_x\left(Z^2(n+1)\right)+\mathbb{E}_x\left(Z^4(n+1)\right)\right)\,(\mu P((j-1)k))(dx)
\\&\qquad\qquad  \leq 1+2(\tilde{A}_2+\tilde{B}_2
)+\tilde{A}_4+\tilde{B}_4=:D.
\end{aligned}
\end{align}
Further, applying the Markov inequality and referring to Remark \ref{rem:2}, as well as to the definition of $M(t)$ (given in \eqref{def:M}), we obtain
\begin{align*}
\int_E\mathbb{P}_x\left(|M(n)|\geq \varepsilon\sqrt{kl}\right)&(\mu P((j-1)k)(dx)
\leq\frac{\mathbb{E}_{\mu P((j-1)k)}|M(n)|}{\varepsilon\sqrt{kl}}\\
&\leq \frac{\|g\|_{\text{BL}}}{\gamma\varepsilon\sqrt{kl}}
\left(\left\langle\mathcal{C},\mu P((j-1)k+n)\right\rangle
+\left\langle\mathcal{C},\mu P((j-1)k)\right\rangle\right)+\frac{\left\|\bar{g}\right\|_{\infty}}{\varepsilon\sqrt{kl}}n\\
&=\frac{\|g\|_{\text{BL}}}{\gamma\varepsilon\sqrt{kl}}
\left(\left\langle P((j-1)k+n)\mathcal{C},\mu \right\rangle+\left\langle P((j-1)k)\mathcal{C},\mu \right\rangle\right)+\frac{\left\|\bar{g}\right\|_{\infty}}{\varepsilon\sqrt{kl}}n.
\end{align*}
Then, choosing $A_1,B_1\geq 0$ (and $\Gamma_1>0$) so that the assertion of Lemma \ref{lem:P(t)C} is valid, we can estimate the second integral as follows:
\begin{align}
\begin{split}\label{eq:last_lemma_4}
\int_E\mathbb{P}_x&\left(|M(n)|\geq \varepsilon\sqrt{kl}\right)\,(\mu P((j-1)k)(dx)\\
&\leq\frac{2\|g\|_{\text{BL}}}{\gamma\varepsilon\sqrt{kl}}\left(A_1\left\langle V^{1/2},\mu\right\rangle+B_1\right)+\frac{\left\|\bar{g}\right\|_{\infty}}{\varepsilon\sqrt{kl}}n=\frac{1}{\sqrt{kl}}\left(\hat{A}+\hat{B}n\right),
\end{split}
\end{align}
where $\hat{A}=2\|g\|_{\text{BL}}(\gamma\varepsilon)^{-1}(A_1\langle V^{1/2},\mu\rangle+B_1)$ and $\hat{B}=\|\bar{g}\|_{\infty}\varepsilon^{-1}$. Finally, combining  \eqref{eq:last_lemma_2} with \eqref{eq:last_lemma_3} and \eqref{eq:last_lemma_4} gives
\begin{align*}
\frac{1}{kl}\sum_{j=1}^l\sum_{n=0}^{k-1}\left\langle G_{k,l,n}, \,\mu P((j-1)k)\right\rangle
& \leq \frac{1}{kl}\sum_{n=0}^{k-1}\sum_{j=1}^{l}D^{1/2}\left(\frac{1}{\sqrt{kl}}\left(\hat{A}+\hat{B}n\right)\right)^{1/2}\\
& \leq \left(kl\right)^{-1/4}D^{1/2}\left(\hat{A}+\hat{B} (k-1)\right)^{1/2}
\end{align*}
for any $k,l\in\n$. This, in turn, implies that
$$
\lim_{k\to\infty}\limsup_{l\to\infty} \frac{1}{kl}\sum_{j=1}^l\sum_{n=0}^{k-1}\left\langle G_{k,l,n}, \,\mu P((j-1)k)\right\rangle=0,
$$
which completes the proof.

\end{proof}

\subsection{Finalization of the proof}

Armed with the results obtained in the previous subsections, we are now in a position to
finalize the proof of the main theorem.
\begin{proof}[Proof of Theorem \ref{thm:main}]
First of all, according to Lemma \ref{lem:inv_exist}, $\{P(t)\}_{t\in\mathbb{R}_+}$ admits a unique probability measure $\mu_*$, which belongs to $\mathcal{M}_{1,2}^V(E)$. Further, taking into account identity \eqref{eq:martingale+rest} and Lemma \ref{lem:rest}, along with insights from Lemma \ref{lem:suffices} and Remark \ref{rem:m1}, it suffices to prove that the CLT holds (with $\sigma^2$ given by \eqref{def:sigma}) for the sequence $\{M(n)\}_{n\in\n_0}$, determined by \eqref{def:M}, whenever the initial distribution $\mu$ of $\Psi$ belongs to $\mathcal{M}_{1,2}^V(E)$. Since, due to Lemma~\ref{lem:martingale}, $\{M(n)\}_{n\in\n_0}$ is a~martingale and $\sigma^2<\infty$, the proof of that  reduces to verifying conditions \ref{cnd:m1}-\ref{cnd:m3} of Theorem~\ref{thm:martingale} (with $\pr=\pr_{\mu}$). This, in turn, has been done in Lemmas \ref{lem:M1}--\ref{lem:M3}, respectively. The~proof of Theorem~\ref{thm:main} is therefore complete.
\end{proof}

\section{A representation of $\sigma^2$}\label{sec:4}
{
In this section, we provide a relatively simple representation of the variance $\sigma^2$, involved in~\hbox{Theorem~\ref{thm:main}}. The line of the reasoning presented below draws heavily from ideas of~\cite{Bhattacharya}.

Suppose that the semigroup $\{P(t)\}_{t\in\mathbb{R}_+}$ enjoys conditions \ref{(h0)}-\ref{(h2)} with a continuous function \hbox{$V:E\to\mathbb{R}_+$}, and let $\mu_*$ denote its unique invariant probability measure. \hbox{Further}, consider the space $\mathbb{L}:=\mathcal{L}^2(\mu_*)$ of all Borel measurable and $\mu_*$-square integrable functions from $E$ to $\mathbb{R}$ (precisely, the corresponding quotient space under the relation of~$\mu_*$-a.e. \hbox{equality}), endowed with the norm\vspace{-0.1cm}
$$\norma{f}_2:=\<f^2,\mu_*\>^{1/2} \quad\text{for}\quad f\in\mathbb{L}.\vspace{-0.1cm}$$
Since, given $p\in\{1,2\}$, $f\in \mathbb{L}$, and $t\in\mathbb{R}_+$, we have $\<P(t)(|f|^p),\mu_*\>=\<|f|^p,\mu_*\><\infty$, it~follows that $\langle |f|^p, \delta_xP(t)\rangle =P(t)(|f|^p)(x)<\infty$ for $\mu_*\text{-a.e. } x\in E$. Thus, bearing in mind \eqref{def:Pf}, we can identify $P(t)f^p$ as a~real-valued Borel measurable function that coincides with the map $x\mapsto \<f^p,\delta_x P(t)\>$ on some Borel set of full measure $\mu_*$ where $P(t)(|f|^p)<\infty)$, and is zero outside this set. Moreover, accounting for this identification and the fact that\vspace{-0.1cm}
$$\<(P(t)f)^2,\mu_*\>\leq \<P(t)f^2,\mu_*\>=\<f^2,\mu_*\><\infty,$$
we see that $P(t)f\in\mathbb{L}$ and $\norma{P(t)f}_2\leq \norma{f}_2$. Consequently, $\{P(t)\}_{t\in\mathbb{R}_+}$ can be viewed as a~contraction semigroup on $\mathbb{L}$.

Now, let $\mathbb{L}_0$ denote the center of the semigroup $\{P(t)\}_{t\in\mathbb{R}_+}$ on $\mathbb{L}$, that is
}
$$\mathbb{L}_0:=\left\{f\in\mathbb{L}:\, \lim_{t\to 0^+} \norma{P(t) f - f}_2=0 \right\}.$$
Then, the infinitesimal generator $A$ of $\{P(t)\}_{t\in\mathbb{R}_+}$ can be defined on the domain
$$D_A:=\left\{f\in\mathbb{L}_0:\, \lim_{t\to 0^+} \norma{\frac{P(t) f - f}{t}-\hat{f}\,}_2=0\;\;\text{for some}\;\;\hat{f}\in\mathbb{L}_0 \right\}$$
by
$$
Af:=\lim_{t\to 0^+}\frac{P(t) f - f}{t}\;\;\text{in}\;\;\norma{\cdot}_2\quad\text{for}\quad f\in D(A).
$$

Before we state the main result of this section, let us observe that the corrector function~$\chi$, given by \eqref{def:chi}, belongs to $\mathbb{L}$. { To see this, it suffices to note that Remark~\ref{rem:2} (applied with $\mu=\mu_*$), together with Corollary~\ref{lem:sup_t}, yields that
$$
\<\chi^2, \mu_*\>=\ew_{\mu_*}\left(\chi^2(\Psi(0)) \right)\leq  \left(\frac{\norma{g}_{\mathrm{BL}}}{\gamma}\right)^2 \<\mathcal{C}^2,\mu_*\><\infty.
$$}
Having established this, we can now prove the following:
\begin{theorem}\label{thm:rep_sigma}
Suppose that $\{P(t)\}_{t\in\mathbb{R}_+}$ satisfies hypotheses \ref{(h0)}--\ref{(h2)} with some \hbox{continuous} \hbox{$V:E\to\mathbb{R}_+$}. Then $\sigma^2$, given by \eqref{eq:sigma}, takes the form
\begin{equation}\label{e:rep_sigma}
\sigma^2=-2\<\chi A \chi,\mu_*\>=2\< \chi\bar{g},\mu_*\>,
\end{equation}
whenever $g\in \operatorname{Lip}_b(E)$, involved in \eqref{def:chi}, is such that, for every $x\in E$, the map $t\mapsto P(t)g(x)$ is continuous at $t=0$. 
\end{theorem}
\begin{proof}
First of all, observe that $\chi\in D_A$ and $A\chi=-\bar{g}$. To show this, note that
\begin{align*}
\norma{\frac{P(t)\chi - \chi}{t}-(-\bar{g})\,}_2^2&=\int_E \left(\frac{1}{t}\left(\int_t^{\infty} P(s)\bar{g}(x)\,ds-\int_0^{\infty} P(s)\bar{g}(x)\,ds \right) +\bar{g}(x)\right)^2\mu_*(dx)\\
&=\int_E \left(\bar{g}(x)-\frac{1}{t}\int_0^t P(s)\bar{g}(x)\,ds \right)^2\mu_*(dx)\quad\text{for all}\quad t>0.
\end{align*}
In view of the continuity at $0$ of $P(\cdot)g(x)$, $x\in E$, we have (see, e.g., \cite[Lemma 5.5.1]{heil})
$$\lim_{t\to 0^+} \frac{1}{t}\int_0^t P(s)\bar{g}(x)\,ds=\bar{g}(x)\quad\text{for any}\quad x\in E.$$
Hence, using the Lebesgue dominated convergence theorem, we can conclude that
$$\lim_{t\to 0^+} \norma{\frac{P(t)\chi - \chi}{t}-(-\bar{g})\,}_2^2=0,$$
which implies the desired claim.

Now, for each $n\in\n$, consider the sequence $\{Z_n(k)\}_{k\leq n}$ of the increments of $\{M(k/n)\}_{k\leq n}$, defined by
$$Z_n(k):=M\left(\frac{k}{n}\right)-M\left(\frac{k-1}{n}\right)\quad\text{for}\quad k\in\{1,\ldots,n\}.$$
Let $n\in\n$ be arbitrarily fixed. Since $\{M(t)\}_{t\in\mathbb{R}_+}$ is a martingale with respect to the natural filtration of $\Psi$, and $\mu_*$ is an invariant distribution of $\Psi$, it follows that $\{Z_n(k)\}_{k\leq n}$ forms a~sequence of pairwise orthogonal and identically distributed random variables on $(\Omega,\mathcal{F}_{\infty},\pr_{\mu_*})$ with mean $0$. In view of this, we have 
\begin{equation}\label{eq:dc0}
\sigma^2=\ew_{\mu_*}\left(M^2(1)\right)=\sum_{k=1}^n \ew_{\mu_*}\left(Z_n^2(k)\right)=n\ew_{\mu_*}\left(Z_n^2(1)\right)=n\ew_{\mu_*}\left(M^2\left(\frac{1}{n}\right)\right),
\end{equation}
and the expectation on the right-hand side can be expressed as
\begin{align}\label{eq:dc1}
\begin{split}
\ew_{\mu_*}\left(M^2\left(\frac{1}{n}\right)\right)&
=\ew_{\mu_*}\left( \left(\chi(\Psi(1/n))-\chi(\Psi(0))\right)^2\right)+\ew_{\mu_*}\left(\left(\int_0^{1/n}\bar{g}(\Psi(s))ds \right)^2\right)\\
&\quad+2\ew_{\mu_*}\left(\left(\chi(\Psi(1/n))-\chi(\Psi(0))\right)\int_0^{1/n}\bar{g}(\Psi(s))ds\right).
\end{split}
\end{align}

Let us define
$$\epsilon_n:=A\chi-\frac{P(1/n)\,\chi-\chi}{1/n}.$$
Then, keeping in mind \eqref{eq:Ef} and taking into account that
$$P(1/n)\chi=\chi+\frac{1}{n}A\chi-\frac{1}{n}\epsilon_n,$$
we can write
\begin{align*}
\ew_x\left( \left(\chi(\Psi(1/n))-\chi(\Psi(0))\right)^2\right)&=P(1/n)\chi^2(x)+\chi^2(x)-2\chi(x)P(1/n)\chi(x)\\
&=P(1/n)\chi^2(x)+\chi^2(x)-2\chi(x)\left(\chi(x)+\frac{1}{n}A\chi(x)-\frac{1}{n}\epsilon_n(x) \right)\\
&=P(1/n)\chi^2(x)-\chi^2(x)-\frac{2}{n}\chi(x)A\chi(x)+\frac{2}{n}\chi(x)\epsilon_n(x)
\end{align*}
{for $\mu_*\text{-a.e. } x\in E$}. This, together with the invariance of $\mu_*$, shows that the first term on the right-hand side of \eqref{eq:dc1} can be expressed as
\begin{equation}
\label{eq:dc2}
\ew_{\mu_*}\left( \left(\chi(\Psi(1/n))-\chi(\Psi(0))\right)^2\right)=-\frac{2}{n}\left(\<\chi A\chi,\mu_*\>-\<\chi \epsilon_n,\mu_*\>\right).
\end{equation}
Hence, putting
$$
\delta_n:=\ew_{\mu_*}\left(\left(\int_0^{1/n}\bar{g}(\Psi(s))ds \right)^2\right),\quad \Delta_n:=\ew_{\mu_*}\left(\left(\chi(\Psi(1/n))-\chi(\Psi(0))\right)\int_0^{1/n}\bar{g}(\Psi(s))ds  \right),
$$
we can write \eqref{eq:dc1} as
$$\ew_{\mu_*}\left(M^2\left(\frac{1}{n}\right)\right)=-\frac{2}{n}\left(\<\chi A\chi,\mu_*\>-\<\chi \epsilon_n,\mu_*\>\right)+\delta_n+2\Delta_n.$$
Consequently, \eqref{eq:dc0} then takes the form
$$\sigma^2=-2\left(\<\chi A\chi,\mu_*\>-\<\chi \epsilon_n,\mu_*\>\right)+n\delta_n+2n\Delta_n.$$

{Since $\chi,\epsilon_n\in\mathbb{L}$} for every $n\in\n$, we can apply the Cauchy-Schwarz inequality to obtain
$$|\<\chi \epsilon_n,\mu_*\>|\leq \norma{\chi}_2\norma{\epsilon_n}_2\quad\text{for any}\quad n\in\n,$$
which, together with the fact that $\lim_{n\to\infty}\norma{\epsilon_n}_2=0$, yields that
\begin{equation}\label{eq:dc3}
\lim_{n\to\infty} \<\chi \epsilon_n,\mu_*\>=0.
\end{equation}

What is now left is to prove that also the sequences $(n\delta_n)_{n\in\n}$ and $(n\Delta_n)_{n\in\n}$ tend to $0$ as $n\to\infty$. Note that, for every $n\in\mathbb{N}$, we have
\begin{equation}\label{eq:dc4}
\delta_n\leq \frac{1}{n^2}\norma{\bar{g}}_{\infty}^2,
\end{equation}
which means that $0\leq n\delta_n\leq n^{-1}\norma{\bar{g}}_{\infty}^2$ for all $n\in\n$, whence $n\delta_n\to 0$ as $n\to\infty$. Moreover, using sequentially the Cauchy-Schwarz inequality, \eqref{eq:dc2} and \eqref{eq:dc4}, we can conclude that, for each $n\in\n$,
\begin{align*}
|\Delta_n|&\leq \left(\ew_{\mu_*}\left( \left(\chi(\Psi(1/n))-\chi(\Psi(0))\right)^2\right)\right)^{1/2}\left(\ew_{\mu_*}\left(\left(\int_0^{1/n}\bar{g}(\Psi(s))ds \right)^2\right) \right)^{1/2}\\
&=\left(-\frac{2}{n}\left(\<\chi A\chi,\mu_*\>-\<\chi \epsilon_n,\mu_*\>\right)\right)^{1/2}\delta_n^{1/2}
\leq \frac{\sqrt{2}\norma{\bar{g}}_{\infty}}{n^{3/2}}\left(-\<\chi A\chi,\mu_*\>+\<\chi \epsilon_n,\mu_*\>\right)^{1/2},
\end{align*}
which gives
$$|n\Delta_n|\leq \frac{\sqrt{2}\norma{\bar{g}}_{\infty}}{n^{1/2}}\left(-\<\chi A\chi,\mu_*\>+\<\chi \epsilon_n,\mu_*\>\right)^{1/2} \quad\text{for all}\quad n\in\n.$$
This, in conjunction with \eqref{eq:dc3}, shows that $n\Delta_n\to 0$ as $n\to\infty$, and therefore the proof is complete.
\end{proof}
{
\section{A note on the functional CLT in the stationary case}\label{sec:fclt}

Although investigating the functional central limit theorem (FCLT), also known as \emph{Donsker's invariance principle}, is not the primary focus of this paper, it is worth to note that, under the employed assumptions, \hbox{\cite[Theorem 2.1]{Bhattacharya}} implies the validity of such a theorem in the case where $\Psi$ is stationary.

To present a relevant result, upon assuming $\mu_*$ to be the unique invariant probability measure of $\{P(t)\}_{t\in\mathbb{R}_+}$, and given $g\in \operatorname{Lip}_b(E)$, consider the sequence $\big\{\Gamma_g^{(n)}\big\}_{n\in\n}$ of stochastic processes defined by
$$\Gamma_{\bar{g}}^{(n)}(t):=\frac{1}{\sqrt{n}}\int_0^{nt} \bar{g}(\Psi(s))\,ds \quad\text{for}\quad t\geq 0,\;n\in\n, $$
with $\bar{g}=g-\<g,\mu_*\>$. 
Further, let $\operatorname{C}_0(\mathbb{R}_+)$ be the space of all real-valued continuous functions on $\mathbb{R}_+$ starting at zero, endowed with the topology of uniform convergence on compact sets. Since $g$ is bounded, it is easily seen that the processes $\Gamma_{\bar{g}}^{(n)}$, $n\in\n$, can be regarded as random variables with values in $\operatorname{C}_0(\mathbb{R}_+)$. Apart from that, let $\mathbb{W}_{\sigma}$ stand for the Wiener measure on (the Borel $\sigma$-field~of) $\operatorname{C}_0(\mathbb{R}_+)$ with zero drift and variance $\sigma^2$. 

When $\mu$ is the initial distribution of $\Psi$, the process $\{\bar{g}(\Psi(t))\}_{t\in\mathbb{R}_+}$ is said to obey the FCLT if, for some $\sigma\geq 0$, the distributions of $\Gamma_{\bar{g}}^{(n)}$ on $\operatorname{C}_0(\mathbb{R}_+)$ (under $\pr_{\mu}$) converge weakly to the Wiener measure $\mathbb{W}_{\sigma}$, i.e.,
\begin{equation}
\label{e:fclt}
\lim_{n\to \infty} \ew_{\mu} \,h\left( \Gamma_{\bar{g}}^{(n)}\right)=\int_{\operatorname{C}_0(\mathbb{R}_+)}h\,d\mathbb{W}_{\sigma} \quad\text{for any bounded continuous}\;\;  h:\operatorname{C}_0(\mathbb{R}_+)\to\mathbb{R}.
\end{equation}

\begin{proposition}\label{prop:FCLT}
Suppose that $\{P(t)\}_{t\in\mathbb{R}_+}$ satisfies hypotheses \ref{(h0)}--\ref{(h2)} with some continuous \hbox{$V:E\to\mathbb{R}_+$}, and that $t\mapsto P(t)g(x)$ is continuous at $t=0$ for every $x\in E$. Furthermore, assume that $\Psi$ is progressively measurable and stationary, i.e., its initial \hbox{distribution} is the (unique) invariant one $\mu_*$. Then, for every $g\in \operatorname{Lip}_b(E)$, the process $\{\bar{g}(\Psi(t))\}_{t\in\mathbb{R}_+}$ obeys the FCLT with $\sigma^2$ given by \eqref{e:rep_sigma}, that is, \eqref{e:fclt} holds with $\mu=\mu_*$.
\end{proposition}
\begin{proof}
First of all, by Lemma \ref{lem:inv_exist}, the semigroup $\{P(t)\}_{t\in\mathbb{R}_+}$ admits a unique invariant measure $\mu_*\in\mathcal{M}_1(E)$ and, therefore, $\Psi$ is ergodic in the sense of \cite{Bhattacharya}; i.e., $\pr_{\mu_*}(F)\in\{0,1\}$ for every $F\in\mathcal{F}_{\infty}$ satisfying $\Theta_t^{-1}(F)=F$ for all $t>0$.  Further, let $A: D_A\to\mathbb{L}_0$ be the~infinitesimal generator of $\{P(t)\}_{t\in\mathbb{R}_+}$ defined in Section \ref{sec:4}, and let $\chi$ denote the corrector function, specified by \eqref{def:chi}. From the first step of the proof of Theorem~\ref{thm:rep_sigma} it follows that $-\chi\in D_A$ and \hbox{$\bar{g}=A(-\chi)$}. This, in particular, means that $\bar{g}$ is in the range of $A$. Hence, we now see that the assertion follows directly from \cite[Theorem 2.1]{Bhattacharya}.
\end{proof}
}
\section{An example of application to some PDMPs}\label{sec:ex}
{We shall end this paper by applying our main result, i.e., Theorem \ref{thm:main} (and Proposition~\ref{prop:FCLT}) to establish the CLT (and the FCLT in the stationary case) for the PDMPs considered in~\cite{ergodic_pdmp}}. More specifically, we will be concerned with a~process involving a deterministic motion governed by a finite number of semiflows, which is punctuated by random jumps, occurring in independent and exponentially distributed time intervals $\Delta\tau_n$ with the same rate $\lambda$. The state right after a jump will depend randomly on the one immediately preceding this jump, and its probability distribution will be
governed by an arbitrary transition law~$J$. 

Let $(Y,\rho_Y)$ be a complete separable metric space, and let $I$ be a finite set, endowed with the discrete metric $\dd$, i.e., $\dd(i,j)=1$ for $i\neq j$ and $\dd(i,j)=0$ otherwise. Moreover, put $X:=Y\times I$, $\bar{X}:=X\times\mathbb{R}_+$, and let $\rho_{X,c}$ denote the metric in $X$ defined by
\begin{equation}\label{def:rho_c}
\rho_{X,c}(x_1,x_2)=\rho_Y(y_1,y_2)+c\,\dd(i_1,i_2)\quad\text{for}\quad x_1=(y_1,i_1),\,x_2=(y_2,i_2)\in X,
\end{equation}
where $c$ is a given positive constant.

Further, suppose that we are given an arbitrary stochastic kernel $J$ on $Y\times \mathcal{B}(Y)$ and some constant $\lambda>0$. In addition to that, consider a stochastic matrix $\{\pi_{ij}\}_{i,j\in I}\subset\mathbb{R}_+$ such that $\min_{i\in I} \pi_{ij_0}>0$ for some $j_0\in I$ and a collection $\{S_i\}_{i\in I}$ of (jointly) continuous semiflows from $\mathbb{R}_+\times Y$ to~$Y$. By saying that $S_i$ is a \emph{semiflow} we mean as usual that
$$S_i(s,S_i(t,y))=S_i(s+t,y)\quad\text{and}\quad S_i(0,y)=y\quad\text{for any}\quad s,t\in\mathbb{R}_+,\;y\in Y.$$ 
Obviously, in practical applications, one usually deals with semiflows generated by unique solutions to certain particular Cauchy problems for autonomous differential equations (see, e.g., \cite[Examples~7.2, 7.3]{ergodic_pdmp}), and their properties are investigated through the operators involved in these equations (such as, e.g., smooth vector fields on $\mathbb{R}^d$ in \cite[\S 5]{b:benaim1} or \cite[\S 4]{b:benaim2}).

We shall investigate a stochastic process $\Psi:=\{(Y(t),\xi(t))\}_{t\in\mathbb{R}_+}$, evolving on the space $(X,\rho_{X,c})$ in such a way that
\begin{equation}\label{def:Pt}
Y(t)=S_{\xi(\tau_n)}(t-\tau_n, Y(\tau_n))\;\;\;\text{and}\;\;\;\xi(t)=\xi(\tau_n)\quad\text{for}\quad t\in [\tau_n,\tau_{n+1}),\;n\in\n_0,
\end{equation}
and $\bar{\Phi}:=\{(Y(\tau_n),\xi(\tau_n),\tau_n)\}_{n\in\n_0}=\{(\Psi(\tau_n),\tau_n)\}_{n\in\n_0}$ is an $\bar{X}$-valued time-homogeneous Markov chain with one-step transition law $\bar{P}:\bar{X}\times\mathcal{B}(\bar{X})\to [0,1]$ given by
\begin{equation}
\label{def:P_bar}
\bar{P}((y,i,s), \bar{A})=\sum_{j\in I} \pi_{ij} \int_0^{\infty} \lambda e^{-\lambda t}\int_Y \mathbbm{1}_{\bar{A}}(u,j,t+s)\,J(S_i(t,y),du)\,dt
\end{equation}
for any $y\in Y$, $i\in I$, $s\in\mathbb{R}_+$ and $\bar{A}\in\mathcal{B}(\bar{X})$.

Obviously, $\Phi:=\{\Psi(\tau_n)\}_{n\in \n_0}$, $\{\xi(\tau_n)\}_{n\in\n_0}$ and $\{\tau_n\}_{n\in\n_0}$ are then also Markov chains w.r.t. their own natural filtrations, and, for every $n\in\mathbb{N}_0$, their transition laws satisfy 
\begin{gather}
\pr_{\bar{\mu}}(\Psi(\tau_{n+1})\in A\,|\,\Psi(\tau_n)=x)
=\bar{P}\left((x,0),A\times\mathbb{R}_+\right)
\quad\text{for}\quad x\in X,\; A\in\mathcal{B}(X),\nonumber\\
\pr_{\bar{\mu}}(\xi(\tau_{n+1})=j\,|\,\xi(\tau_n)=i)=\pi_{ij}\quad\text{for}\quad i,j\in I,\nonumber\\
\label{e:tau}\pr_{\bar{\mu}}(\tau_{n+1}\leq t\,|\,\tau_n=s)=\mathbbm{1}_{[s,\infty)}(t)\left(1-e^{-\lambda(t-s)}\right)\quad\text{for}\quad s,t\in\mathbb{R}_+,
\end{gather}
where $\bar{\mu}=\mu\otimes\delta_0$, and $\mu\in\mathcal{M}_1(X)$ stands for the initial measure of $\Phi$ (and thus of $\Psi$). Importantly, \eqref{e:tau} implies that the increments $\Delta\tau_n:=\tau_{n}-\tau_{n-1}$, $n\in\n$, form a sequence of independent and exponentially distributed random variables with the same rate $\lambda$, and therefore $\tau_n\uparrow\infty$, as $n\to\infty$, $\pr_{\bar{\mu}}$-a.s. (which, in turn, yields that \eqref{def:Pt} is well-defined).

It is not hard to check that $\Psi$, defined as above, is a (piecewise-deterministic) time-homogeneous Markov process. Clearly, such a process is {progressively} measurable, as it has right-continuous sample paths. In the remainder of the paper, $\{P(t)\}_{t\in\mathbb{R}_+}$ will stand for the transition semigroup of this process. Let us emphasize here that, since $\{P(t)\}_{t\in\mathbb{R}_+}$ is stochastically continuous (at $t=0$) by \cite[Lemma 5.1 (iii)]{ergodic_pdmp}, it indeed satisfies also the last condition required in the definition of transition semigroup adopted in this paper, i.e., the map
$(t,x)\mapsto P(t)(x,A)$ is $\mathcal{B}(\mathbb{R}_+\times E)/\mathcal{B}(\mathbb{R})$-measurable (see \cite[Proposition 3.4.5]{worm_thesis}).

In \cite{ergodic_pdmp}, we have proposed conditions (J1), (J2) on the kernel $J$ and (S1)-(S3) on the semiflows~$S_i$ under which $\{P(t)\}_{t\in\mathbb{R}_+}$ enjoys hypothesis \ref{(h1)} with $\rho=\rho_{X,c}$, specified by~\eqref{def:rho_c}, and $V$ given by
\begin{equation}\label{def:V}
V(x):=\rho_{Y}(y,y^*)\quad\text{for any}\quad x:=(y,i)\in X,
\end{equation}
with some arbitrarily fixed $y^*\in Y$, provided that the constants involved in  \hbox{\cite[(J1) and (S2)]{ergodic_pdmp}} are interrelated by inequality \hbox{\cite[(4.8)]{ergodic_pdmp}}, and that $c$ is sufficiently large. More precisely, this follows from \hbox{\cite[Lemma 6.2]{ergodic_pdmp}}, applied together with \cite[Proposition 7.1]{ergodic_pdmp}. Moreover, statement (i) of \cite[Lemma 5.1]{ergodic_pdmp} yields that, if $J$ is Feller, then so is $\{P(t)\}_{t\in\mathbb{R}_+}$, i.e., \ref{(h0)}~holds.

To apply Theorem \ref{thm:main} {(and Proposition \ref{prop:FCLT})}, it therefore suffices to verify when \ref{(h2)} is met. We will show that this hypothesis does hold upon assuming the following conditions:\vspace{-0.1cm}
\begin{enumerate}[label=\textnormal{(J1')}]
\item\label{j1'} There exist $a,b\geq 0$ for which
$$J\rho_Y^2(\cdot,y^*)(y)\leq a \rho_Y^2(y,y^*)+b\quad \text{for}\quad y\in Y,$$
\end{enumerate}
\begin{enumerate}[label=\textnormal{(S0)}]
\item\label{s0} There exist $R,M\geq 0$ and $N\in \n$ such that
$$\rho_Y(S_i(t,y), y^*)\leq R\,\rho_Y(y,y^*)+M\left(t^N+1\right)\quad\text{for}\quad y\in Y,\;t\geq 0,\; i\in I;\vspace{-0.1cm}$$\vspace{-0.8cm}
\end{enumerate}
with $2aR^2<1$. Obviously, \cite[(J1)]{ergodic_pdmp} can be derived from \ref{j1'} by using the H\"older inequality, which shows that the former holds with $\tilde{a}:=\sqrt{a}$ and $\tilde{b}:=\sqrt{b}$. At the end of this section, we will also demonstrate how to link (S0) with the aforementioned hypotheses \cite[(S1)-(S3)]{ergodic_pdmp}.

\begin{lemma}\label{lem:gen-lap}
Suppose that \ref{j1'} and \ref{s0} hold with $a,R\geq 0$ such that $2aR^2<1$. Then there exist constants ${\Gamma}>0$ and $C\geq 0$ such that, for every $t_0\geq 0$ and the function \hbox{$U_{t_0}:\bar{X}\to \mathbb{R}_+$} given by
\begin{equation}
\label{df:u_t}
U_{t_0}(x,t):=e^{-\lambda (t_0-t)}V^2(x)\mathbbm{1}_{[0,t_0]}(t)\quad\text{for}\quad x\in X,\;t\geq 0,
\end{equation}
we have
$$
\sum_{n=0}^{\infty} \bar{P}^n U_{t_0}(x,0)\leq e^{-{\Gamma} t_0} V^2(x)+C\quad\text{for all} \quad x\in X.
$$
\end{lemma}
\begin{proof}
Fix $t_0\geq 0$, and define $\eta:=2a R^2<1$ and $m:=2N$ with $N$ given in \ref{s0}. Then, using sequentially \ref{j1'}, \ref{s0} and \eqref{simple_fact}, we infer that, for any $(y,i)\in X$ and~$t\geq 0$,
\begin{align}
\begin{split}
\label{lem:e1}
J\rho_Y^2(\cdot,y^*)(S_i(t,y))&\leq a\rho_Y^2(S_i(t,y),y^*)+b
\leq 2 a\left(R^2\rho_Y^2(y,y^*)+2M^2\left(t^{2N}+1\right)\right)+b\\
&= 2 a R^2 \rho_Y^2(y,y^*)+4 a M^2 t^m+(4a M^2+b)\\
&\leq \eta \rho_Y^2(y,y^*)+D\left(\frac{t^m}{m!}+1\right)\quad\text{with}\quad D:=4 a M^2 m!+b\geq 0.
\end{split}
\end{align}\vspace*{-0.3cm}

In what follows, we will show inductively that for every $n\in\n$\vspace{-0.1cm}
\begin{align}
\begin{split}
\label{lem:e2}
\bar{P}^n U_{t_0}(y,i,s)&\leq \lambda^n e^{\lambda(t_0-s)}\left( \left(\eta^n \rho_Y^2(y,y^*)+D\sum_{k=0}^{n-1}\eta^k \right)\frac{(t_0-s)^n}{n!}\right.\\
&\quad\left.+D\left(\sum_{k=0}^{n-1}\eta^k\right)\frac{(t_0-s)^{m+n}}{(m+n)!}\right)\mathbbm{1}_{[0,t_0]}(s)\quad\text{for all}\quad (y,i,s)\in \bar{X}.
\end{split}
\end{align}
For $n=1$ and any $(y,i,s)\in \bar{X}$ we get
\begin{align*}
\bar{P}U_{t_0}(y,i,s)&=\sum_{j\in I} \pi_{ij}\int_0^{\infty}\lambda e^{-\lambda t} \int_Y U_{t_0}(u,j,s+t)\,J(S_i(t,y),du)\,dt\\
&=\int_0^{\infty}\lambda e^{-\lambda t}\int_Y e^{-\lambda(t_0-s-t)}\rho_Y^2(u,y^*)\mathbbm{1}_{[0,t_0]}(s+t)J(S_i(t,y),du)\,dt\\
&=\lambda e^{-\lambda(t_0-s)}\int_0^{\infty} J\rho_Y^2(\cdot,y^*)(S_i(t,y))\mathbbm{1}_{[0,t_0]}(s+t)\,dt,
\end{align*}
and therefore from \eqref{lem:e1} it follows that, for $s\leq t_0$,
\begin{align*}
\bar{P}U_{t_0}(y,i,s)&\leq \lambda e^{-\lambda(t_0-s)}\int_0^{t_0-s}\left(\eta\rho_Y^2(y,y^*)+D\left(\frac{t^m}{m!}+1 \right) \right)\,dt\\
&=\lambda e^{-\lambda(t_0-s)}\left(\eta\rho_Y^2(y,y^*)(t_0-s)+D\frac{(t_0-s)^{m+1}}{(m+1)!} + D(t_0-s)\right)\\
&=\lambda e^{-\lambda(t_0-s)}\left(\left(\eta\rho_Y^2(y,y^*)+D \right)(t_0-s) +D\frac{(t_0-s)^{m+1}}{(m+1)!}\right),
\end{align*}
and $\bar{P}U_{t_0}(y,i,s)=0$ for $s>t_0$.  Hence \eqref{lem:e2} holds with $n=1$. Now, suppose that~\eqref{lem:e2} is fulfilled with some arbitrarily fixed $n\in\n$. Then, by identity $\bar{P}^{n+1}U_{t_0}=\bar{P}(\bar{P}^n U_{t_0})$ and the induction hypothesis, we obtain
\begin{align*}
\bar{P}^{n+1}U_{t_0}(y,i,s)&\leq \int_0^{\infty} \lambda e^{-\lambda t}\int_Y\lambda^n e^{-\lambda(t_0-s-t)}\left(\left(\eta^n \rho_Y^2(u,y^*)+D\sum_{k=0}^{n-1}\eta^k\right)\frac{(t_0-s-t)^n}{n!}\right.\\
&\left.\quad+D\left(\sum_{k=0}^{n-1}\eta^k\right)\frac{(t_0-s-t)^{m+n}}{(m+n)!}\right)\mathbbm{1}_{[0,t_0]}(s+t)J(S_i(t,y),du)\,dt\\
&=\lambda^{n+1} e^{-\lambda(t_0-s)}\int_0^{\infty}\left(\left(\eta^n J\rho_Y^2(\cdot,y^*)(S_i(t,y))+D\sum_{k=0}^{n-1}\eta^k\right)\frac{(t_0-s-t)^n}{n!}\right.\\
&\left.\quad+D\left(\sum_{k=0}^{n-1}\eta^k\right)\frac{(t_0-s-t)^{m+n}}{(m+n)!}\right)\mathbbm{1}_{[0,t_0]}(s+t)\,dt.
\end{align*}
Consequently, using again \eqref{lem:e1}, for $s\leq t_0$, we get
\begin{align*}
\bar{P}^{n+1}&U_{t_0}(y,i,s)\leq\lambda^{n+1} e^{-\lambda(t_0-s)}\int_0^{t_0-s}\left(\left(\eta^n \left(\eta\rho_Y^2(y,y^*)+D\left(\frac{t^m}{m!}+1\right) \right)+D\sum_{k=0}^{n-1}\eta^k\right)\right.
\\
&\left.\quad\times \frac{(t_0-s-t)^n}{n!}+D\left(\sum_{k=0}^{n-1}\eta^k\right)\frac{(t_0-s-t)^{m+n}}{(m+n)!}\right)\,dt\\
&=\lambda^{n+1} e^{-\lambda(t_0-s)}\left(\left(\eta^{n+1} \rho_Y^2(y,y^*)+D\sum_{k=0}^n\eta^k\right)\frac{1}{n!}\int_0^{t_0-s}(t_0-s-t)^n\,dt\right.\\
&\left.\quad+\frac{D\eta^n}{m!\,n!}\int_0^{t_0-s}t^m(t_0-s-t)^n\,dt+D\left(\sum_{k=0}^{n-1}\eta^k\right)\frac{1}{(m+n)!}\int_0^{t_0-s}(t_0-s-t)^{m+n}\,dt \right).
\end{align*}
Further, taking into account that
$$\int_0^T t^k(T-t)^l\,dt=\frac{k!\,l!\,T^{k+l+1}}{(k+l+1)!}\quad \text{for all}\quad k,l\in\n_0,\;T>0,$$
we can finalize the above estimation (in the case where $s\leq t_0$) as follows:
\begin{align*}
&\bar{P}^{n+1}U_{t_0}(y,i,s)\leq \lambda^{n+1} e^{-\lambda(t_0-s)}\left(\left(\eta^{n+1} \rho_Y^2(y,y^*)+D\sum_{k=0}^n\eta^k\right)\frac{(t_0-s)^{n+1}}{(n+1)!}\right.\\
&\left.\quad+\frac{D\eta^n (t_0-s)^{m+n+1}}{(m+n+1)!}+D\left(\sum_{k=0}^{n-1}\eta^k\right)\frac{(t_0-s)^{m+n+1}}{(m+n+1)!}\right)\\
&=\lambda^{n+1} e^{-\lambda(t_0-s)}\left(\left(\eta^{n+1} \rho_Y^2(y,y^*)+D\sum_{k=0}^n\eta^k\right)\frac{(t_0-s)^{n+1}}{(n+1)!}+D\left(\sum_{k=0}^{n}\eta^k\right)\frac{(t_0-s)^{m+n+1}}{(m+n+1)!}\right).
\end{align*}
Obviously, $\bar{P}^{n+1}U_{t_0}(y,i,s)=0$ for $s>t_0$. According to the induction principle, we can therefore conclude that \eqref{lem:e2} indeed holds for all $n\in\n$.

Now, observe that inequality \eqref{lem:e2} applied with $s=0$ gives
\begin{align*}
\bar{P}^n U_{t_0}(y,i,0)&\leq \lambda^n e^{-\lambda t_0}\left(\left(\eta^n\rho_Y^2(y,y^*)+D\sum_{k=0}^{n-1}\eta^k \right)\frac{t_0^n}{n!}+D\left(\sum_{k=0}^{n-1}\eta^k\right)\frac{t_0^{m+n}}{(m+n)!} \right)\\
&\leq e^{-\lambda t_0} \left( \frac{(\eta\lambda t_0)^n}{n!}\rho_Y^2(y,y^*)+\frac{D}{1-\eta}\left(\frac{(\lambda t_0)^n}{n!}+\frac{1}{\lambda^m} \cdot\frac{(\lambda t_0)^{m+n}}{(m+n)!}\right) \right)
\end{align*}
for all $(y,i)\in X$ and $n\in\n_0$ (for $n=0$, this follows trivially from the definition of $U_{t_0}$). Finally, we obtain
\begin{align*}
\sum_{n=0}^{\infty} \bar{P}^n U_{t_0}(y,i,0)&\leq e^{-\lambda t_0}\left(\sum_{n=0}^{\infty} \frac{(\eta\lambda t_0)^n}{n!}\rho_Y^2(y,y^*)+\frac{D}{1-\eta}\left(1+\frac{1}{\lambda^m}\right)\sum_{n=0}^{\infty} \frac{(\lambda t_0)^n}{n!} \right)\\
&=e^{-\lambda t_0}\left(e^{\eta\lambda t_0}\rho_Y^2(y,y^*) +\frac{D}{1-\eta}\left(1+\frac{1}{\lambda^m}\right)e^{\lambda t_0} \right)\\
&=e^{-{{\Gamma}} t_0}\rho_Y^2(y,y^*)+{C}=e^{-{{\Gamma}} t_0}V^2(y,i)+C\quad\text{for all}\quad (y,i)\in X,
\end{align*}
where ${{\Gamma}}:=(1-\eta)\lambda>0$ and $ {C}:=D(1-\eta)^{-1}\left(1+\lambda^{-m}\right)\geq 0$, which completes the proof.
\end{proof}

\begin{proposition}\label{thm:lap}
Suppose that \ref{j1'} and \ref{s0} hold with $a,R\geq 0$ such that $2aR^2<1$.  Then $\{P(t)\}_{t\in\mathbb{R}_+}$ enjoys hypothesis \ref{(h2)} with $V$ given by \eqref{def:V}.
\end{proposition}
\begin{proof}
Let ${{\Gamma}}>0$ and $C\geq 0$ be the constants for which the assertion of Lemma \ref{lem:gen-lap} is valid. Further, fix $t\geq 0$, $x=(y,i)\in X$, and put $\bar{x}:=(x,0)$. Then, we can write
\begin{equation}
\label{e:3}
P(t)V^2(x)=\ew_{\bar{x}}\left(V^2(\Psi(t))\right)=\sum_{n=0}^{\infty} \ew_{\bar{x}}\left(V^2(\Psi(t))\mathbbm{1}_{\{\tau_n\leq t<\tau_{n+1}\}}\right).
\end{equation}

Now, let $\{\mathcal{F}_n\}_{n\in\n_0}$ denote the natural filtration of the chain $\Phi$, and fix an arbitrary $n\in\n_0$. Taking into account that
$\pr_{\bar{x}}(\tau_{n+1}>t\,|\,\mathcal{F}_n)=\pr_x(\tau_{n+1}>t\,|\,\tau_n)=e^{-\lambda(t-\tau_n)}$ on the set $\{\tau_n\leq t\}$ (which follows from \eqref{e:tau}), we get
\begin{align*}
\ew_{\bar{x}}\left(V^2(\Psi(t))\mathbbm{1}_{\{\tau_n\leq t<\tau_{n+1}\}}\,|\,\mathcal{F}_n\right)&=\ew_{\bar{x}}\left(\rho_Y^2\left(S_{\xi_n}(t-\tau_n,Y_n),y^*\right) \mathbbm{1}_{\{\tau_n\leq t<\tau_{n+1}\}} \,|\,\mathcal{F}_n\right)\\
&=\pr_{\bar{x}}(\tau_{n+1}>t\,|\,\mathcal{F}_n)\,\rho_Y^2\left(S_{\xi_n}(t-\tau_n,Y_n),y^*\right)\mathbbm{1}_{\{\tau_n\leq t\}}\\
&=e^{-\lambda(t-\tau_n)}\rho_Y^2\left(S_{\xi_n}(t-\tau_n,Y_n),y^*\right)\mathbbm{1}_{\{\tau_n\leq t\}}.
\end{align*}
Further, letting $m:=2N$, and using \ref{s0} and \eqref{simple_fact}, we can conclude that
\begin{align*}
\ew_{\bar{x}}&\Big(V^2(\Psi(t))\mathbbm{1}_{\{\tau_n\leq t<\tau_{n+1}\}}\,|\,\mathcal{F}_n\Big)\\
&\leq 2e^{-\lambda(t-\tau_n)} \left(R^2\rho_Y^2(Y_n,y^*)+2M^2\left((t-\tau_n)^{2N}+1\right)  \right) \mathbbm{1}_{\{\tau_n\leq t\}}\\
&\leq 2R^2 e^{-\lambda(t-\tau_n)}\rho_Y^2(Y_n,y^*)\mathbbm{1}_{[0,t]}(\tau_n)+4M^2e^{-\lambda(t-\tau_n)}\left((t-\tau_n)^m+1\right)\mathbbm{1}_{[0,t]}(\tau_n).
\end{align*}
Taking the expectation of both sides of this inequality gives
\begin{align}
\label{e:4}
\begin{split}
\ew_{\bar{x}}\left(V^2(\Psi(t))\mathbbm{1}_{\{\tau_n\leq t<\tau_{n+1}\}}\right)&\leq 2 R^2 \ew_{\bar{x}}\left(e^{-\lambda(t-\tau_n)}\rho_Y^2(Y_n,y^*) \mathbbm{1}_{[0,t]}(\tau_n)\right)\\
&\quad+4 M^2\ew_{\bar{x}}\left(e^{-\lambda(t-\tau_n)}\left((t-\tau_n)^m+1\right)\mathbbm{1}_{[0,t]}(\tau_n)\right).
\end{split}
\end{align}
If we now sum up both sides of \eqref{e:4} over all $n\in\n_0$, then, returning to \eqref{e:3}, we can deduce that
\begin{align}
\label{e:3b}
\begin{split}
P(t) V^2(x&)\leq 2 R^2\sum_{n=0}^{\infty} \ew_{\bar{x}}\left(e^{-\lambda(t-\tau_n)}\rho_Y^2(Y_n,y^*)\mathbbm{1}_{[0,t]}(\tau_n)\right)\\
&\quad+4 M^2\left(\sum_{n=1}^{\infty}\ew_{\bar{x}}\left(e^{-\lambda(t-\tau_n)}\left((t-\tau_n)^m+1\right)\mathbbm{1}_{[0,t]}(\tau_n)\right)+e^{-\lambda t}\left(t^m+1\right)\right).
\end{split}
\end{align}

On the other hand, it follows from Lemma \ref{lem:gen-lap} that
\begin{align}
\label{e:5}
\begin{split}
\sum_{n=0}^{\infty}\ew_{\bar{x}}\left(e^{-\lambda(t-\tau_n)}\rho_Y^2(Y_n,y^*)\mathbbm{1}_{[0,t]}(\tau_n)\right)&=\sum_{n=0}^{\infty}\ew_{\bar{x}}\left( U_t(\Phi_n,\tau_n)\right)=\sum_{n=0}^{\infty}\bar{P}^nU_t(\bar{x})\\
&\leq e^{-{{\Gamma}} t}V^2(x)+C,
\end{split}
\end{align}
where $U_t:\bar{X}\to\mathbb{R}_+$ is given by \eqref{df:u_t} (with $t$ in place of $t_0$).
Moreover, having in mind that~$\tau_n$ has the Erlang distribution, we get
\begin{align}
\label{e:6}
\sum_{n=1}^{\infty}\ew_{\bar{x}}&\left(e^{-\lambda(t-\tau_n)}\left((t-\tau_n)^m+1\right)\mathbbm{1}_{[0,t]}(\tau_n)\right)=\sum_{n=1}^{\infty}\int_0^t e^{-\lambda(t-s)}\left((t-s)^m+1\right) e^{-\lambda s}\frac{\lambda^n s^{n-1}}{(n-1)!}\,ds \nonumber\\
&=\lambda e^{-\lambda t}\int_0^t e^{\lambda s}\left((t-s)^m+1\right)\,ds=\lambda e^{-\lambda t}\int_0^t e^{\lambda(t-u)}\left(u^m+1\right) \,du\\
&=\lambda \int_0^t e^{-\lambda u}\left(u^m+1\right)\,du\leq \lambda \int_0^{\infty}e^{-\lambda u} \left(u^m+1\right)\,du=\frac{m!}{\lambda^m}+1<\infty.\nonumber
\end{align}

Finally, applying \eqref{e:3b}-\eqref{e:6}, we infer that
\begin{align*}
P(t)V^2(x)&\leq 2 R^2\left(e^{-{{\Gamma}} t}V^2(x)+C\right)+4M^2\left(\frac{m!}{\lambda^m}+1+e^{-\lambda t}\left(t^m+1\right)\right)\leq A e^{-{{\Gamma}} t}V^2(x)+ B 
\end{align*}
with
$A=2R^2$ and $B:=2\left(R^2 {C}+2M^2\left(m!\lambda^{-m}+1 \right)+2M^2\sup_{t \geq 0}e^{-\lambda t} \left(t^m+1\right)\right)$.
\end{proof}

Let us now consider two conditions constituting strengthened forms of \cite[(S1)]{ergodic_pdmp}, namely:
\begin{enumerate}[label=\textnormal{(S1*)}]
\item\label{s1*} There exist $M\geq 0$ and $N\in\n$ such that
$$ \max_{i\in I}\sup_{y\in Y}\rho_Y(S_i(t,y), y)\leq M\left(t^N+1\right)\quad\text{for}\quad t\geq 0;\vspace{-0.1cm}$$\vspace{-0.8cm}
\end{enumerate}
\begin{enumerate}[label=\textnormal{(S1')}]
\item\label{s1'} There exist $M\geq 0$ and $N\in\n$ such that
$$\max_{i\in I} \rho_Y(S_i(t,y^*), y^*)\leq M\left(t^N+1\right)\quad\text{for}\quad t\geq 0.\vspace{-0.1cm}$$\vspace{-0.8cm}
\end{enumerate}
In addition to that, recall condition \cite[(S2)]{ergodic_pdmp}:
\begin{enumerate}[label=\textnormal{(S2)}]
\item\label{s2} There exist $L\geq 1$ and $\alpha<\lambda$ such that
$$\rho_Y(S_i(t,y_1),S_i(t,y_2))\leq Le^{\alpha t}\rho_Y(y_1,y_2)\quad\text{for}\quad t\geq 0,\;y_1,y_2\in Y,\;i\in I.\vspace{-0.1cm}$$
\end{enumerate}

\begin{remark} \label{rem:implications}
Using the triangle inequality, it is easy to verify the following implications: 
$$
\begin{tikzcd}[arrows=Rightarrow, sep=small]
  & \text{\cite[(S1) $\wedge$ (S3)]{ergodic_pdmp}}\\
\ref{s1*} \arrow[ur, "\varphi(t)=\varphi_1(t) \atop \mathcal{L}(y)=\mathcal{L}_1(y)"] \arrow[dr, "R=1"'] & &  \arrow[ul, "\varphi(t)=\varphi_2(t) \atop \mathcal{L}(y)=\mathcal{L}_2(y)"'] \ref{s1'} \arrow[dl, "R=L"] & \hspace{-0.7cm}\wedge \; [\ref{s2} \text{ with } \alpha\leq 0],\\
& \ref{s0}
\end{tikzcd}
$$\vspace{-0.2cm}
where \cite[(S3)]{ergodic_pdmp} is obtained (depending on the assumptions) with
$$\varphi_1(t):=2M(t^N+1),\; \mathcal{L}_1(y)=1,\;\text{or}\;\; \varphi_2(t)=2(Le^{\alpha t}+M(t^N+1)), \;\;\mathcal{L}_2(y)=\rho_Y(y,y^*)+1.$$
\end{remark}
\vspace{-0.1cm}
It is worth stressing here that conditions \ref{s1'} and \ref{s2} with $\alpha\leq 0$ and $L=1$ are met, e.g., by the semiflows genereted by a wide class of dissipative differential equations in Hilbert spaces. If the operators involved in such equations are bounded, then even \ref{s1*} holds. This is explained in detail in \hbox{\cite[Remark 4.3]{ergodic_pdmp}}, which, in turn, is based on \hbox{\cite[Chapter 5.2]{ito}}.

Formulating the main result of this section, apart from hypotheses \ref{s1*}, \ref{s1'}, \ref{s2} and \ref{j1'}, we shall also use condition \cite[(J2)]{ergodic_pdmp}, which ensures the existence of a Markovian coupling of $J$ with certain specific properties. Let us quote it to make our result self-contained. Below, a \emph{substochastic kernel} will mean a map defined similarly to a stochastic kernel (see Section \ref{sec:markov_operators}), but allowed to be a subprobability measure with respect to the second variable.

\begin{enumerate}[label=\textnormal{(J2)}]
\item\label{j2}There exists a substochastic kernel $Q_J: Y^2\times\mathcal{B}(Y^2)\to [0,1]$ satisfying
$$
Q_J((y_1,y_2), B\times Y)\leq J(y_1,B)\quad\text{and}\quad Q_J((y_1,y_2), Y\times B)\leq J(y_2,B)
$$
for any $y_1,y_2\in Y$ and $B\in\mathcal{B}(Y)$ such that, for certain constants $\tilde{a}, l\geq 0$, we have
\begin{gather*}
\int_{Y^2}\rho_Y(u_1,u_2)\,Q_J((y_1,y_2),du_1\times du_2)\leq \tilde{a}\rho_Y(y_1,y_2)\quad\text{for any}\quad y_1,y_2\in Y,\\
\inf_{(y_1,y_2)\in Y} Q_J((y_1,y_2), \{(u_1,u_2)\in Y^2:\, \rho_Y\big(u_1,u_2)\leq \tilde{a}\rho_Y(y_1,y_2) \}\big)>0,\\
Q_J((y_1,y_2), Y^2)\geq 1-l\rho_Y(y_1,y_2) \quad \text{for any}\quad y_1,y_2\in Y.
\end{gather*}
\end{enumerate}

We can now state the announced central convergence criterion for the PDMP under consideration.

\begin{theorem}\label{cor:clt}
Suppose that the kernel $J$ is Feller, \ref{j1'} holds (with certain $a,b\geq 0$), and that \ref{j2} is satisfied with $\tilde{a}:=\sqrt{a}$. Moreover, assume that one of the following statements is fulfilled:
\begin{enumerate}[label=\textnormal{(\roman*)}]
\item\label{cor-clt:i} \ref{s1*}  and \ref{s2} hold with $\alpha<\lambda$ and $L\geq 1$ such that $a<\min\{L^{-2}(1-\alpha\lambda^{-1})^2,\, 2^{-1}\}$,
\item\label{cor-clt:ii} \ref{s1'} and \ref{s2} is satisfied with  $\alpha\leq 0$ and $L\geq 1$ such that $a<(2L^2)^{-1}$.

\end{enumerate}
{Then the assertions of both Theorem \ref{thm:main} and Proposition \ref{prop:FCLT} are valid} for the semigroup  $\{P(t)\}_{t\in\mathbb{R}_+}$ determined by \eqref{def:Pt} and \eqref{def:P_bar}, with $V$ given by \eqref{def:V} and $\sigma^2$ satisfying \eqref{e:rep_sigma}, provided that the constant $c$, involved in \eqref{def:rho_c}, is sufficiently large. 
\end{theorem}

\begin{proof}
As mentioned earlier, \ref{j1'} implies \hbox{\cite[(J1)]{ergodic_pdmp}} with $\tilde{a}=\sqrt{a}$, and  \ref{j2} is just assumed. Further, according to Remark \ref{rem:implications}, each of hypotheses \ref{cor-clt:i}, \ref{cor-clt:ii} guarantees that \hbox{\cite[(S1)-(S3)]{ergodic_pdmp}} hold with $\alpha,L$ satisfying the inequality $\tilde{a}L+\alpha/\lambda<1$, which coincides with that assumed in \hbox{\cite[(4.8)]{ergodic_pdmp}}. As has already been said, this, together with the Feller property of $J$, suffices for hypotheses \ref{(h0)} and \ref{(h1)} to hold. On the other hand, appealing again to Remark~\ref{rem:implications}, we also know that \ref{s0} is satisfied with $R=1$ in case \ref{cor-clt:i} or with $R=L$ in case \ref{cor-clt:ii}, and $2aR^2<1$ in each of these cases. Consequently, Proposition \ref{thm:lap} yields that hypothesis \ref{(h2)} holds as well. Hence, Theorem \ref{thm:main} is indeed valid for the semigroup being considered. 

{
Furthermore, since $\{P(t)\}_{t\in\mathbb{R}_+}$ is stochastically continuous (at $t=0$) by statement (iii) of \hbox{\cite[Lemma 5.1]{ergodic_pdmp}}, it follows that $\sigma^2$ can be represented as indicated in Theorem \ref{thm:rep_sigma}, and that Proposition \ref{prop:FCLT} is also applicable to the given semigroup.}
\end{proof}

\begin{remark} \label{rem:cost_c}
By saying that $c$ should be \emph{sufficiently large} we mean that it supposed to be grater than some specific value dependent on the constants and functions involved in conditions \ref{j1'}, \ref{s2} and \cite[(S3)]{ergodic_pdmp} (which ensures that \cite[Proposition 7.1]{ergodic_pdmp} is valid). To be more precise, letting $\varphi:\mathbb{R}_+\to\mathbb{R}_+$ and $\mathcal{L}:Y\to\mathbb{R}_+$ be the functions for which (S3) holds (see Remark~\ref{rem:implications}),  one may require that
$$c\geq \frac{\lambda-\alpha}{L}\left(M_{\mathcal{L}}{K}_{\varphi}+\frac{M_{\mathcal{L}}M_{\varphi}}{\lambda}\right)+1,$$
where $K_{\varphi}:=\int_0^{\infty} \varphi(t)e^{-\lambda t}\,dt$, $M_{\varphi}:=\sup_{t\in\mathbb{T}}\varphi(t)$ with $\mathbb{T}\subset[0,\infty)$ being an arbitrary bounded interval of positive length for which\, $\sup_{t\in\mathbb{T}}e^{\alpha t}\leq \lambda(\lambda-\alpha)^{-1}$, and $M_{\mathcal{L}}:=\sup_{\rho_Y(y_*,y)<r}\mathcal{L}(y)$ with $r:=4b(1-a)^{-1}$ and
$$\bar{a}:=\frac{\sqrt{a}\lambda L}{\lambda-\alpha},\quad \bar{b}=\sqrt{a}\lambda\left(\frac{MN!}{\lambda^{N+1}}+\frac{1}{\lambda} \right)+ \sqrt{b}.$$
\end{remark}

An important example of the PDMPs under consideration are those with $J$ defined as the transition law of a random iterated function system. In this case $J$ takes the form
$$
J(y,B)=\int_{\Theta}\mathbbm{1}_B(w_{\theta}(y))\,p_{\theta}(y)\,\vartheta(d\theta)\quad\text{for}\;\; y\in Y,\;B\in\mathcal{B}(Y),
$$
where $\{w_{\theta}:\,\theta\in\Theta\}$ is an arbitrary family of continuous transformations from $Y$ to itself, indexed by the elements of some topological measure space $(\Theta,\vartheta)$, and $\{p_{\theta}:\,\theta\in \Theta\}$ is an associated family of state-dependent densities with respect to $\vartheta$. In \hbox{\cite[Proposition 7.2]{ergodic_pdmp}}, we have provided a set of conditions guaranteeing that $J$ of this form satisfies hypotheses \hbox{\cite[(J1) and (J2)]{ergodic_pdmp}} with $Q_J$ given by
$$Q_J((y_1,y_2),C):=\int_{\Theta}\mathbbm{1}_C(w_{\theta}(y_1),w_{\theta}(y_2))\min\{p_{\theta}(y_1),p_{\theta}(y_2)\}\,\vartheta(d\theta)$$
for any $y_1,y_2\in Y$ and $C\in\mathcal{B}(Y^2)$. These conditions can be easily modified to also ensure \ref{j1'}, as required in {Theorem \ref{cor:clt}}. The details are left to the reader.

\section*{Acknowledgments}
The authors thank the anonymous referees for their constructive comments that have helped to significantly improve the paper.
The work of H.W.-\'S. is supported by the project \emph{Near-term Quantum Computers: challenges, optimal implementations and applications} under grant no. POIR.04.04.00-00-17C1/18-00, which is carried out within the Team-Net programme of the
Foundation for Polish Science, co-financed by the European Union under the European Regional Development Fund.

\bibliographystyle{plain}
\bibliography{references}

\begin{thebibliography}{10}

\bibitem{b:benaim2}
M.~Bena{\"i}m, S.~Le~Borgne, F.~Malrieu, and P.-A. Zitt.
\newblock Quantitative ergodicity for some switched dynamical systems.
\newblock {\em Electronic Communications in Probability}, 17(0), 2012.

\bibitem{b:benaim1}
M.~Bena{\"i}m, S.~Le~Borgne, F.~Malrieu, and P.-A. Zitt.
\newblock Qualitative properties of certain piecewise deterministic
  \text{Markov} processes.
\newblock {\em Ann. Inst. Henri Poincar\'e Probab.}, 51(3):1040--1075, 2015.

\bibitem{Bhattacharya}
R.N. Bhattacharya.
\newblock On the functional central limit theorem and the law of the iterated
  logarithm for {M}arkov processes.
\newblock {\em Wahrscheinlichkeitstheorie verw Gebiete}, 60:185--201, 1982.

\bibitem{billingsley86}
P.~Billingsley.
\newblock {\em {Probability and measure, 2nd ed.}}
\newblock John Wiley and Sons, Inc., New York, 1986.

\bibitem{blumenthal}
R.M. Blumenthal and R.K. Getoor.
\newblock {\em {M}arkov processes and potential theory}.
\newblock Pure and Applied Mathematics, 29. Academic Press, Inc., 1968.

\bibitem{bogachev}
V.I. Bogachev.
\newblock {\em Measure theory}, volume~II.
\newblock Springer-Verlag, Berlin, 2007.

\bibitem{brown71}
B.M. Brown.
\newblock {Martingale central limit theorems}.
\newblock {\em Ann. Math. Stat.}, 42(1):59 -- 66, 1971.

\bibitem{cloez_hairer}
B.~Cloez and M.~Hairer.
\newblock {Exponential ergodicity for Markov processes with random switching}.
\newblock {\em Bernoulli}, 21(1):505 -- 536, 2015.

\bibitem{ergodic_prop}
D.~Czapla, K.~Horbacz, and H.~Wojew\'odka-\'Sci\k{a}\.zko.
\newblock Ergodic properties of some piecewise-deterministic {M}arkov process
  with application to gene expression modelling.
\newblock {\em Stoch. Proc. Appl.}, 130(5):2851--2885, 2020.

\bibitem{clt}
D.~Czapla, K.~Horbacz, and H.~Wojew\'odka-\'Sci\k{a}\.zko.
\newblock A useful version of the central limit theorem for a general class of
  {M}arkov chains.
\newblock {\em J. Math. Anal. Appl.}, 484(1):123725, 2020.

\bibitem{lil}
D.~Czapla, K.~Horbacz, and H.~Wojew\'odka-\'Sci\k{a}\.zko.
\newblock The {S}trassen invariance principle for certain non-stationary
  {M}arkov-{F}eller chains.
\newblock {\em Asymptot. Anal.}, 121(1):1--34, 2021.

\bibitem{ergodic_pdmp}
D.~Czapla, K.~Horbacz, and H.~Wojew\'odka-\'Sci\k{a}\.zko.
\newblock Exponential ergodicity in the bounded-{L}ipschitz distance for some
  piecewise-deterministic {M}arkov processes with random switching between
  flows.
\newblock {\em Nonlinear Anal.}, 215:112678, 2022.

\bibitem{dawid-asia}
D.~Czapla and J.~Kubieniec.
\newblock Exponential ergodicity of some {M}arkov dynamical systems with
  application to a {P}oisson driven stochastic differential equation.
\newblock {\em Dyn. Syst.}, 34(1):130--156, 2019.

\bibitem{derrenic_lin01}
Y.~Derriennic and M.~Lin.
\newblock The central limit theorem for {M}arkov chains with normal transition
  operators, started at a point.
\newblock {\em Probab. Theory Related Fields}, 119(4):508--528, 2001.

\bibitem{dudley}
R.~Dudley.
\newblock Convergence of {B}aire measures.
\newblock {\em Studia Math.}, 27:251--268, 1966.

\bibitem{ethier}
N.E. Ethier and T.G. Kurtz.
\newblock {\em Markov processes. Characterization and convergence}.
\newblock John Wiley {\&} Sons, Inc., Hoboken, New Jersey, 1986.

\bibitem{gordin_lifsic78}
M.I. Gordin and B.A. Lif\v{s}ic.
\newblock Central limit theorem for stationary {M}arkov processes.
\newblock {\em Dokl. Akad. Nauk SSSR}, 239(4):766--767, 1978.

\bibitem{gordin_lifsic81}
M.I. Gordin and B.A. Lif\v{s}ic.
\newblock A remark about a {M}arkov process with normal transition operator.
\newblock {\em Third Vilnius Conf. Proba. Stat., Akad. Nauk Litovsk, Vilnius},
  1:147--148, 1981.

\bibitem{gulgowski}
J.~Gulgowski, S.C. Hille, T.~Szarek, and M.A. Ziemla\'nska.
\newblock Central limit theorem for some non-stationary {M}arkov chains.
\newblock {\em Studia Math.}, 246:109--131, 2019.

\bibitem{hairer}
M.~Hairer.
\newblock Exponential mixing properties of stochastic {PDE}s through asymptotic
  coupling.
\newblock {\em Probab. Theory Related Fields}, 124, 2002.

\bibitem{heil}
C.~Heil.
\newblock {\em Introduction to Real Analysis}.
\newblock Springer, 2019.

\bibitem{holzmann05}
H.~Holzmann.
\newblock The central limit theorem for stationary {M}arkov processes with
  normal generator -- with applications to hypergroups.
\newblock {\em Stochastics}, 77(4):371--380, 2005.

\bibitem{ito}
K.~Ito and F.~Kappel.
\newblock {\em {Evolution equations and approximations, vol. 61 of Ser. Adv.
  Math. Appl. Sci.}}
\newblock World Scientific, New Jersey, 2002.

\bibitem{subgeom}
R.~Jin and A.~Tan.
\newblock Central limit theorems for {M}arkov chains based on their convergence
  rates in {W}asserstein distance.
\newblock {\em Preprint available at \url{https://arxiv.org/abs/2002.09427}},
  2020.

\bibitem{kapica_sleczka}
R.~Kapica and M.~\'Sl\k{e}czka.
\newblock Random iteration with place dependent probabilities.
\newblock {\em Probab. Math. Statist.}, 40(1):119--137, 2020.

\bibitem{kipnis_varadhan86}
C.~Kipnis and S.R.S. Varadhan.
\newblock Central limit theorem for additive functionals of reversible {M}arkov
  processes and applications to simple exclusions.
\newblock {\em Comm. Math. Phys.}, 104:1--19, 1986.

\bibitem{kom_szar_peszat}
T.~Komorowski, S.~Peszat, and T.~Szarek.
\newblock Passive tracer in a flow corresponding to two-dimensional stochastic
  navier-stokes equations.
\newblock {\em Nonlinearity}, 26(7):1999--2026, 2013.

\bibitem{kom_walczuk}
T.~Komorowski and A.~Walczuk.
\newblock Central limit theorem for {M}arkov processes with spectral gap in the
  {W}asserstein metric.
\newblock {\em Stoch. Proc. Appl.}, 122(5):2155--2184, 2012.

\bibitem{krengel}
U.~Krengel.
\newblock {\em Ergodic theorems. With a supplement by Antoine Brunel.}
\newblock De Gruyter Studies in Mathematics, 6. Walter de Gruyter \& Co.,
  Berlin, New York, 1985.

\bibitem{lasota_from_fractals}
A.~Lasota.
\newblock From fractals to stochastic differential equations.
\newblock In P.~Garbaczewski, M.~Wolf, and A.~Weron, editors, {\em Chaos -- The
  Interplay Between Stochastic and Deterministic Behaviour (Proceedings of the
  XXXIst Winter School of Theoretical Physics, Karpacz, Poland, 13-24 Feb
  1995)}, volume 457 of {\em Lecture Notes in Physics}, pages 235--255, Berlin,
  1995. Lecture Notes in Physics 457, Springer Verlag.

\bibitem{levy35}
P.~L\'evy.
\newblock Propri\'et\'es asymptotiques des sommes de variables ind\'ependantes
  ou enchain\'ees.
\newblock {\em Journal des math\'ematiques pures et appliqu\'ees. Series 9.},
  14(4):347--402, 1935.

\bibitem{loeve77}
M.~Lo\`eve.
\newblock {\em {Probability Theory 1, 4th ed.}}
\newblock Springer-Verlag, New York, 1977.

\bibitem{mw}
M.~Maxwell and M.~Woodroofe.
\newblock Central limit theorems for additive functionals of {M}arkov chains.
\newblock {\em Ann. Probab.}, 28:713--724, 2000.

\bibitem{mt93}
S.P. Meyn and R.L. Tweedie.
\newblock {\em Markov chains and stochastic stability}.
\newblock Springer-Verlag, Berlin, {H}eidelberg, {N}ew {Y}ork, 1993.

\bibitem{klo}
S.~Olla, C.~Landim, and T.~Komorowski.
\newblock {\em Fluctuations in {M}arkov processes. {T}ime symmetry and
  martingale approximation}.
\newblock Springer, Berlin, {H}eidelberg, 2012.

\bibitem{sharpe}
M.~Sharpe.
\newblock {\em General theory of {M}arkov processes. Pure and applied
  mathematics, 133}.
\newblock Academic Press, Inc., 1988.

\bibitem{sleczka}
M.~\'Sl\k{e}czka.
\newblock The rate of convergence for iterated function systems.
\newblock {\em Studia Math.}, 205:201--214, 2011.

\bibitem{Walters}
P.~Walters.
\newblock {\em An introduction to ergodic theory.}
\newblock Graduate Texts in. Mathematics, 79. Springer-Verlag, Berlin,
  Heidelberg, New York, 1982.

\bibitem{hania}
H.~Wojew\'odka.
\newblock Exponential rate of convergence for some {M}arkov operators.
\newblock {\em Stat. Probab. Lett.}, 83(10):2337--2347, 2013.

\bibitem{worm_thesis}
D.T.H. Worm.
\newblock {\em Semigroups on spaces of measures}.
\newblock Leiden University (PhD thesis), Leiden, The Netherlands, 2010.

\end{thebibliography}

\end{document}